\documentclass[a4paper, 12pt] {extarticle}

\usepackage[cp1251]{inputenc}
\usepackage[russian, english]{babel}
\usepackage{amsfonts}
\usepackage{mathtext}
\usepackage{cite}
\usepackage{tikz}
\usetikzlibrary{decorations.pathreplacing,calligraphy}

\usepackage{graphicx}
\usepackage{hyperref}
\usepackage{verbatim}
\usepackage{amsmath}
\usepackage{amsthm}
\usepackage{amssymb}
\usepackage{delarray}
\usepackage{mathrsfs}
\usepackage{xcolor}
\usepackage{transparent}
\usepackage{enumerate}

\newtheorem{theorem}{Theorem}
\newtheorem{lemma}{Lemma}
\newtheorem{theoremA}{Theorem}

\theoremstyle{definition}
\newtheorem{prop}{Proposition}
\newtheorem{remark}{Remark}
\newtheorem{definition}{Definition}

\numberwithin{equation}{section}

\begin{document}

\thispagestyle{empty}

\begin{center}
{\bf\Large On Solutions of Systems of Differential Equations  on Half-Line with Summable Coefficients}
\end{center}

\begin{center}
{\bf\large Maria Kuznetsova\footnote{National Research Saratov State University,  Astrakhanskaya Street 83, Saratov, 410012, Russia. e-mail: {\it kuznetsovama@info.sgu.ru}}}
\end{center}

{\bf Abstract.} We consider a system of differential equations and obtain its solutions with exponential asymptotics and analyticity with respect to the spectral parameter.
Solutions of such type have importance in studying spectral properties of differential operators.
Here, we consider the system of first-order differential equations 
on a half-line with summable coefficients, containing a nonlinear dependence on the spectral parameter.
We obtain fundamental systems of solutions with analyticity in certain sectors, in which it is possible to apply the method of successive approximations. 
We also construct non-fundamental systems of solutions with analyticity in a large sector, including two previously considered neighboring sectors.
The obtained results admit applications in studying inverse spectral problems for the higher-order differential operators with distribution coefficients.
\smallskip



\bigskip
{\color{blue} This is an extended version of the paper accepted for publication in Lobachevskii Journal of Mathematics. 
The present version differs by the detailed proofs of Lemma~\ref{theta lemma}, Theorem~\ref{theta theorem}, and Theorem~\ref{neighboring}.
We also added Section~\ref{n=2} and Appendix~\ref{appendix B}, that are absent in the journal version.
}

\section{Introduction}
In this paper, we study the system of differential equations on the half-line $x \ge 0$
\begin{equation} \label{sys}
\mathrm{\bf y}' = \big(\lambda F(x) + A(x) + C(x, \lambda)\big) \mathrm{\bf y}, \quad \mathrm{\bf y} = [y_j(x)]_{j=1}^n,
\end{equation}
where $\lambda \in \mathbb C_0$ is the spectral parameter, $\mathbb C_0 := \mathbb C \setminus \{ 0\}.$ Here, $F,$ $A,$ and $C$ are matrix functions of order $n$ satisfying the following conditions.
\begin{enumerate}[I.]
\item $F(x) = \rho(x) B,$ where $\rho \in L[0, T)$ for any $T > 0$ and $\rho(x) > 0$ a.e., while $B$ is a constant diagonal matrix:
 $ B= \mathrm{ diag} \{ b_j \}_{j=1}^n,$ $b_j \in \mathbb C$ for $j = \overline{1, n}.$

\item $A(x) =[a_{jk}(x)]_{j,k=1}^n,$ where $a_{jk} \in L[0, \infty)$ for $j,k=\overline{1, n}.$

\item $C(x, \lambda) = [c_{jk}(x,\lambda)]_{j,k=1}^n,$ where each component $c_{jk}(x,\lambda)$ is a holomorphic mapping of $\lambda \in \mathbb C_0$ to $L[0, \infty),$ and
\begin{equation} \label{C con}
\| C(\cdot, \lambda)\|_{L[0, \infty)} := \max_{j,k=\overline{1, n}} \| c_{jk}(\cdot,\lambda) \|_{L[0, \infty)} \to 0, \quad \lambda \to \infty.
\end{equation}
\end{enumerate}
Solutions of system~\eqref{sys} are considered in the class of vector functions $\mathrm{\bf y} = [y_j(x)]_{j=1}^n$ such that $y_j \in AC[0, T]$ for any $T > 0,$ $j=\overline{1, n}.$
 
We aim to obtain solutions of system~\eqref{sys} possessing certain asymptotics and analyticity with respect to the spectral parameter. 
Fundamental systems of solutions (FSS) with the mentioned properties is an important tool in the spectral theory of differential operators, see~\cite{nai,tamarkin,nova,bond, beals,yur}.
Herewith, construction of FSS for the $n$-th order differential equation with the spectral parameter   
\begin{equation} \label{n-th order}
u^{(n)} + p_1(x) u^{(n-1)} + \ldots + p_n(x) u = \lambda^n \rho^n(x) u
\end{equation}
 can be reduced to construction of FSS for first-order system~\eqref{sys}, where
\begin{equation} \label{partic C}
F(x) = \rho(x) \mathrm{diag}\,\{b_j\}_{j=1}^n, \quad b_j = \exp\Big(\frac{2 \pi i j}{n}\Big), \; j=\overline{1, n},\quad C(x, \lambda) = \sum_{\nu=1}^{n-1} C_\nu(x) \lambda^{-j},
\end{equation}
while the elements of $A(x)$ are proportional to $p_1(x).$ 

Systems of first-order differential equations under various assumptions on the coefficients were studied in~\cite{birkhoff,tamarkin,langer,vagabov,malamud,beals,rykh,sav,yurko-sys, sav sad,stud,kosarev, surveys} and other works. 
We distinguish the classical case when, in notations of system~\eqref{sys}, the coefficients of $A$ are absolutely continuous, see~\cite{birkhoff, tamarkin, langer}. In this case, for obtaining the FSS with certain asymptotics and analyticity with respect to the spectral parameter, one can apply the standard method of successive approximations, that requires the related integral operators to be contraction mappings. Then, the FSS are obtained for a sufficiently large $|\lambda| \ge \lambda_0 >0,$ separately for $\lambda$ in each closed sector $\overline{\Gamma_\kappa}$ for which the values $\{ b_j \}_{j=1}^n$ can be renumbered so that 
\begin{equation} \label{numer1}
\mathrm{Re}\, (\lambda b_1) \ge \mathrm{Re}\, (\lambda b_2) \ge \ldots \ge \mathrm{Re}\,(\lambda b_n), \quad \lambda \in \overline{\Gamma_\kappa}.
\end{equation}

Here, we study system~\eqref{sys} in the non-classical case when the coefficients of $A$ are only summable.
It is more difficult, because, in general, the corresponding integral operators are not contraction mappings even for large $\lambda \in \overline{\Gamma_\kappa}.$ 
To obtain the FSS in this case, 
one should apply a modified method of successive approximations~\cite{rykh, sav, sav sad,vagabov}, proving by rigorous estimates that the squares of the integral operators are contraction mappings. The mentioned works~\cite{rykh, sav, sav sad,vagabov} concern systems of differential equations on a finite interval.
In contrast to them, we study system~\eqref{sys} on the half-line $x \ge 0.$ 

The non-classical case is crucial for studying $n$-th order differential equations~\eqref{n-th order} with distribution coefficients, see~\cite{mirz-1,mirz-2,bond}. 
The regularization approach~\cite{mirz-1,mirz-2} allows interpreting such equations as systems~\eqref{sys} with summable coefficients, wherein the matrices $F$ and $C$ have form~\eqref{partic C}. 
In relation with this, in~\cite{sav}, there were constructed the FSS of systems~\eqref{sys} on a finite interval under conditions similar to I--III.
This result was used for obtaining the FSS of equations with distribution coefficients, those, subsequently, were applied in studying spectral properties of the corresponding differential operators, see~\cite{bond, bond2, bond3}.

As far as we know, construction of the FSS in the case of a half-line was considered in detail only in~\cite{stud}, wherein the results were obtained for absolutely continuous coefficients of $A(x).$ That construction was applied for studying inverse spectral problems, see~\cite{yur,yurko-sys}.  Here, we study the more difficult non-classical case. Moreover, we use the specifics of the half-line and for $\lambda \in \overline{\Gamma_\kappa},$ construct a family of the FSS having exponential asymptotics for $x \ge \alpha$ and determined for $|\lambda|\ge\lambda_\alpha>0,$ where $\alpha \ge 0$ is a parameter.
Imposing the condition
\begin{equation} \label{C alpha}
\exists \varphi(\alpha)>0, \; \alpha \ge 0 \colon \  \varphi(\alpha) \to 0,\; \sup_{|\lambda|\ge \varphi(\alpha)} \| C(\cdot, \lambda)\|_{L[\alpha, \infty)} \to 0, \quad \alpha \to \infty,
\end{equation}
we get that $\lambda_\alpha \to 0$ as $\alpha \to \infty.$ Thus, we obtain the FSS analytic at $\lambda \in {\Gamma_\kappa}$ arbitrarily close to $0.$
Condition~\eqref{C alpha} is quite general to include the matrices $C$ in~\eqref{partic C} with $\| C_\nu \|_{L[0, \infty)} < \infty,$ $\nu=\overline{1, n-1},$ that are obtained after regularization of $n$-th order equations with distribution coefficients, see Proposition~\ref{C prop}.

Note that the unboundedness of $x \ge 0$ does not allow us consideration of shifted sectors used in~\cite{sav, nai, sav sad},  containing the boundaries between neighboring sectors $\Gamma_\kappa.$  Instead of this, we introduce large sectors $\Omega_m$ including two neighboring sectors $\Gamma_\kappa$ and construct non-fundamental systems of solutions analytic at $\lambda \in \Omega_m.$ 
The properties of these systems are sufficient for solving the inverse spectral problems considered in~\cite{yur, bond3} in the case of a half-line.
 
The paper is organized as follows. In Section~\ref{prelim}, we introduce necessary objects and justify application of the method of successive approximations, see Lemma~\ref{main lemma}. In Section~\ref{fss sec}, we obtain the FSS with the needed properties, reducing~\eqref{sys} to a system of integral equations, see Theorem~\ref{fss}. 
 Under additional conditions, we prove that the residual members from asymptotics of Theorem~\ref{fss} are square summable with respect to $\lambda,$ see Theorem~\ref{complement}.
 In Section~\ref{large sec}, we construct the non-fundamental systems of solutions analytic at $\lambda$ from the large sector, see Theorem~\ref{neighboring}. 
 In Section~\ref{n=2}, we give application of our results to a second order equation.
Appendix~\ref{appendix A} contains the detailed proof of auxiliary Theorem~\ref{theta theorem}.
 In Appendix~\ref{appendix B}, we provide the basic facts on holomorphic mappings and obtain Proposition~\ref{hol} that is used in the proof of Lemma~\ref{main lemma}.
 
The main results of the paper are Theorems~\ref{fss}--\ref{neighboring}. They can be applied in studying inverse spectral problems for the higher-order operators with distribution coefficients.

\section{Preliminaries} \label{prelim}
Introduce the function $p(x) = \int_0^x \rho(t) \, dt$ and the matrix $D(x) =[d_{jk}(x)]_{j,k=1}^n,$
where 
$$d_{jk}(x) = \left\{\begin{array}{cc}
a_{jk}(x), & b_j = b_k, \\
0, &  b_j \ne b_k,
\end{array}\right.
\quad j,k = \overline{1, n}, \quad x \ge 0.
 $$
For the parameter $\alpha \ge 0,$ consider the matrix-function $M_\alpha(x)= [m_{jk}(x)]_{j,k=1}^n$ of order $n$ that is the solution of the initial-value problem 
\begin{equation} \label{M init value}
M'_\alpha(x) = D(x) M_\alpha(x),\; x \ge 0, \quad M_\alpha(\alpha) = I,
\end{equation}
where $I$ is the identity matrix. Here and below, we denote the dependence on $\alpha$ by the last lower index and omit it for brevity if the obtained notation has two or more indices.

The existence and the uniqueness of $M_\alpha(x)$ along with some its properties  is ensured by the following proposition, being a consequence of the results obtained in~\cite{sav}.
\medskip

\begin{prop} \label{prop M}
There exists a unique solution $M_\alpha(x) = [m_{jk}(x)]_{j,k=1}^n$ of initial-value problem~\eqref{M init value} whose components belong to $AC[0, \infty).$  
It has the following properties.

\begin{enumerate}
\item For each $x \ge 0,$ the matrix $M_\alpha(x)$ is invertible, and the components of the matrix $M^{-1}_\alpha(x) = [\tilde m_{jk}(x)]_{j,k=1}^n$ belong to $AC[0, \infty).$ 

\item If $b_j \ne b_k$ for $j, k \in \overline{1, n},$ then $m_{jk} \equiv \tilde m_{jk} \equiv 0.$

\item \begin{equation} \label{M alpha}
\max_{j,k = \overline{1, n}} \sup_{x \ge \alpha}|m_{jk}(x)| \le e^a, \quad \max_{j,k = \overline{1, n}} \sup_{x \ge \alpha}|\tilde m_{jk}(x)| \le e^a,\quad a := n \| A\|_{L[0, \infty)}.
\end{equation}

\item The matrices $M_\alpha(x)$ and $B$ commute. 

\item If the numbers $b_j,$ $j=\overline{1, n},$ are distinct, then we have $M_\alpha(x) = \mathrm{diag}\{ \exp(\int_\alpha^x  a_{jj}(t) \, dt)\}_{j=1}^n.$
\end{enumerate}
\end{prop}

Introduce the matrices $Q_\alpha(x) = [q_{jk}(x)]_{j,k=1}^n$ and $R_\alpha(x, \lambda) = [r_{jk}(x,\lambda)]_{j,k=1}^n$  as follows:
$$Q_\alpha(x) = M_\alpha^{-1}(x)(A(x) - D(x)) M_\alpha(x),\quad R_\alpha(x, \lambda) = M_\alpha^{-1}(x) C(x, \lambda) M_\alpha(x),$$
where $x \ge 0$ and $\lambda \in \mathbb C_0.$ 
Using property~2 of Proposition~\ref{prop M}, it is easy to see that if $b_j = b_k$ for $j,k \in \overline{1, n},$ then $q_{jk} \equiv 0.$  By~\eqref{M alpha}, we also have 
\begin{equation} \label{QR}
  \| Q_\alpha\|_{L[\alpha, \infty)} \le e^{2a} \| A - D\|_{L[\alpha, \infty)}, \quad \| R_\alpha(\cdot, \lambda)\|_{L[\alpha, \infty)} \le e^{2a}\gamma_\alpha(\lambda), \quad \alpha \ge 0, \; \lambda \in \mathbb C_0,
\end{equation}
 where  we put $\gamma_\alpha(\lambda) = \| C(\cdot, \lambda) \|_{L[\alpha, \infty)}.$
Condition~\eqref{C con} yields that $\gamma_\alpha(\lambda) \to 0$ uniformly on $\alpha \ge 0$ as $\lambda \to \infty.$ 
As well, for any fixed $\alpha \ge 0,$ the function $\gamma_\alpha(\lambda)$ is continuous on $\mathbb C_0,$ since the elements of $C(x, \lambda)$ are holomorphic mappings of $\lambda \in \mathbb C_0$ to $L[0, \infty).$  

Further, we need to study the behavior of certain values depending additionally on the parameters $k \in \overline{1, n}$ and $\omega \in \mathbb C.$ Let $\Lambda \subset \mathbb C_0$ be a domain such that 
\begin{equation} \label{omega numer}
\mathrm{Re}\,(\lambda b_j) \ge \mathrm{Re}\,(\lambda \omega) \ge \mathrm{Re}\,(\lambda b_l), \quad j=\overline{1, k-1}, \; l=\overline{k, n}, \quad \lambda \in \overline{\Lambda}.
\end{equation}
For $j,l \in \overline{1,n}$ and $\alpha \ge 0,$ consider the following functions of $s, x \ge\alpha$ and $\lambda \in \overline{\Lambda}:$
\begin{equation} \left.
\begin{array}{c}
\displaystyle\nu_{jl}(s, x, \lambda) =\int_{t_1}^{t_2} q_{jl}(t) 
\exp\big(\lambda g_{jl}(s,x,t)\big)\, dt,\\[3mm]
\displaystyle\varkappa_{jl}(s, x, \lambda) = \int_{t_1}^{t_2} r_{jl}(t, \lambda) \exp\big(\lambda g_{jl}(s,x,t)\big)\, dt, \quad \lambda \ne 0,
\end{array} \right\}
\label{nu-mu}
\end{equation}
where  $g_{jl}(s,x,t) := (b_l - \omega)(p(t)-p(s)) + (b_j - \omega)(p(x)-p(t))$ and
\begin{equation} \label{rule} 
(t_1, t_2) := \left\{
\begin{array}{cc}
(x, s), & j,l < k, \\
\big(\max\{x,s\}, \infty\big), & j< k \le l, \\
\big(\alpha, \min\{x, s\}\big), & l< k \le j, \\
(s, x), & k \le j, l.
\end{array}\right.
\end{equation}
In~\eqref{nu-mu} and below, we agree that in the case $t_1 \ge t_2,$ an integral $\int_{t_1}^{t_2}$ is considered to be $0.$ 
In the opposite case $t_1 < t_2,$  for the variable of integration  $t$ in~\eqref{nu-mu}, we have that $t \ge s \Leftrightarrow l \ge k$ and $t \le x \Leftrightarrow j \ge k.$ 
Taking into account~\eqref{omega numer} and that $p(x)$ is increasing, we obtain the inequality
\begin{equation} \label{g}
\mathrm{Re} \, \big(\lambda g_{jl}(s,x,t)\big)  \le 0, \quad t_1 < t < t_2, \quad s,x \ge \alpha, \;\lambda \in \overline{\Lambda}.
\end{equation}
In particular, it yields that the integrals in~\eqref{nu-mu} are well-defined as  the Lebesgue integrals.

\begin{lemma} \label{theta lemma}
For $\lambda \in \overline{\Lambda},$ denote
$
\displaystyle\theta_{\alpha}(\lambda) = \sup_{\substack{j,l \in \overline{1, n}; \\ s, x \ge\alpha}} |\nu_{jl}(s, x, \lambda)|.
$
If $\lambda \to \infty,$ then we have
${\theta_\alpha(\lambda) \to 0}$ uniformly on $\alpha.$ 
As well, the following estimates hold:
\begin{equation} \label{est}
\sup_{\substack{j,l = \overline{1, n}; \\ s, x \ge\alpha}} |\varkappa_{jl}(s,x,\lambda)| \le \gamma_\alpha(\lambda) e^{2a}, \quad
\theta_{\alpha}(\lambda) \le a e^{2a}, \quad \lambda \in \overline{\Lambda}\setminus\{ 0\}.
\end{equation}
\end{lemma}
\begin{proof}
Inequalities~\eqref{est} easily follow from~\eqref{QR} and~\eqref{g}. Let us prove $\theta_\alpha(\lambda) \to 0$ as $\overline{\Lambda}\ni \lambda \to \infty$ uniformly on $\alpha \ge 0.$ It is equivalent to that for any $\varepsilon > 0$ and $j,l \in \overline{1, n},$ there exists $\lambda_* > 0$ such that
\begin{equation} \label{uniform con}
|\nu_{jl}(s, x, \lambda)| < \varepsilon, \quad s,x \ge \alpha, \quad |\lambda| \ge \lambda_*,
\end{equation}
and $\lambda_*$ does not depend on $\alpha.$ If $b_j = b_l,$ then $\nu_{jl}= 0,$ and~\eqref{uniform con} is obvious. Let $b_j \ne b_l;$ for definiteness, we assume $j,l<k$ (the other cases in~\eqref{rule} are proceeded analogously). Then, we can consider only $s > x$ (otherwise $\nu_{jl} = 0$), and
$$\nu_{jl}(s, x, \lambda)=\int_x^s q_{jl}(t) \exp(\lambda g_{jl}\big(s, x, t)\big) \, dt. $$
Let $T > 0$ be such that $\| q_{jl} \|_{L[T, \infty)} < \frac{\varepsilon}{2}.$ If $x \ge T,$ we immediately have
$$|\nu_{jl}(s, x, \lambda)| \le \int_x^s |q_{jl}(t)|\,dt \le \int_T^\infty |q_{jl}(t)|\,dt < \frac{\varepsilon}{2}.$$
 If $x < T,$ then
 $$|\nu_{jl}(s, x, \lambda)| \le  \Bigg|\int_{\min\{s, T\}}^s q_{jl}(t)e^{\lambda g_{jl}(s, x, t)} \, dt\Bigg| + |I_1|,\quad I_1 := \int_x^{\min\{s, T\}} q_{jl}(t) e^{\lambda g_{jl}(s, x, t)}\,dt.$$
Here, the first summand either equals $0$ (if $\min\{s, T\} = s$) or does not exceed $\frac{\varepsilon}{2}$ (if $\min\{s, T\} = T$). In any case, we have
 \begin{equation} \label{med1}
|\nu_{jl}(s, x, \lambda)| < |I_1| + \frac{\varepsilon}{2}.
\end{equation}
Changing the variable $t$ on $\xi = p(t)$ in the integral $I_1,$ we obtain
\begin{equation}
I_1 = \int_{l_1}^{l_2} f_{jl}(\xi) e^{\lambda g_{jl}(s, x, p^{-1}(\xi))}\,d\xi, \quad l_1 := p(x), \;  l_2 := p(\min\{s, T \}),
\end{equation}
where $f_{jl}(\xi) = \frac{q_{jl}(p^{-1}(\xi))}{\rho(p^{-1}(\xi))},$ while $p^{-1}$ is the inverse function to $p,$ and $[l_1, l_2] \subseteq [0,  p(T)].$ 
Note that 
$\int_0^{p(T)} |f_{jl}(\xi)|\, d\xi = \int_0^T |q_{jl}(t)| \, dt,$
which yields $f_{jl} \in L[0, p(T)].$
By this reason, there exists a continuously differentiable function $\tilde f_{jl} \in C^{(1)}[0, p(T)]$ such that 
\begin{equation*}
\| f_{jl} - \tilde f_{jl}\|_{L[0, p(T)]} \le \frac{\varepsilon}{4}.
\end{equation*}
Using this inequality and that $\big|e^{\lambda g_{jl}(s, x, p^{-1}(\xi))}\big| \le 1$ as $\xi \in [l_1, l_2],$ we arrive at
\begin{equation} \label{med2}
|I_1| \le  \int_{l_1}^{l_2}|f_{jl}(\xi) - \tilde f_{jl}(\xi)|\, d\xi +|I_2| \le \frac{\varepsilon}{4} +|I_2|, \quad 
I_2 := \int_{l_1}^{l_2} \tilde f_{jl}(\xi) e^{\lambda g_{jl}(s, x, p^{-1}(\xi))}\,d\xi.
\end{equation}
Note that $g_{jl}(s, x, p^{-1}(\xi)) = (b_l - b_j)\xi + p(x)(b_j - \omega) - p(s)(b_l - \omega),$ and 
$$\frac{\partial e^{\lambda g_{jl}(s, x, p^{-1}(\xi))}}{\partial \xi} = (b_l - b_j)\lambda e^{\lambda g_{jl}(s, x, p^{-1}(\xi))}.$$
Applying integration in parts to $I_2,$ we obtain
$$I_2 = \frac{1}{\lambda(b_l - b_j)}\left(\tilde f_{jl}(l_2) e^{\lambda g_{jl}(s, x, p^{-1}(l_2))} - \tilde f_{jl}(l_1) e^{\lambda g_{jl}(s, x, p^{-1}(l_1))} -   \int_{l_1}^{l_2}\tilde f'_{jl}(\xi) e^{\lambda g_{jl}(s, x, p^{-1}(\xi))}\,d\xi\right).$$
Denote $M_{jl}^\varepsilon = 2 \displaystyle \sup_{\xi \in [0, p(T)]} |\tilde f_{jl}(\xi)| + \| \tilde f_{jl}\|_{L[0, p(T)]}.$ The obtained equality yields $|I_2| \le \frac{M_{jl}^\varepsilon}{|\lambda||b_l - b_j|},$ and for $|\lambda| \ge \lambda_* = \frac{4 M_{jl}^\varepsilon}{\varepsilon|b_l - b_j|},$ we have the estimate $|I_2| \le \frac{\varepsilon}{4}.$ Combining it with~\eqref{med1} and~\eqref{med2}, we obtain~\eqref{uniform con}.  
 \end{proof}

Now, we prove a lemma on the solvability of a system of integral equations, applying the method of successive approximations. We
 will use this lemma in the proofs of Theorems~\ref{fss} and~\ref{neighboring}, assigning different values to $k,$ $\omega$ and $\Lambda.$
\begin{lemma}[the main lemma] \label{main lemma}
Let the parameters $k \in \overline{1, n},$  $\alpha \ge 0,$ and $\omega \in \mathbb C$ be fixed, and let a domain $\Lambda \subset \mathbb C_0$ be such that~\eqref{omega numer} holds.
Denote by ${\bf BC}_n$ the Banach space of the vector functions $\mathrm{\bf z}(x) = [z_j(x)]_{j=1}^n$ whose components are bounded and continuous on $[\alpha, \infty),$ equipped with the norm 
$$\| \mathrm{\bf z}\|_{{\bf BC}_n} = \max_{j=\overline{1, n}}\sup_{x \ge \alpha}|z_j(x)|.$$ 
We determine the operator ${\cal V}_k(\lambda)$ that acts on $\mathrm{\bf z}(x)=[z_j(x)]_{j=1}^n \in {\bf BC}_n$ as follows:  
\begin{equation} \label{operat}
{\cal V}_k(\lambda) {\bf z} := [f_j(x)]_{j=1}^n, \quad f_j(x) = \left\{\begin{array}{cc}
\displaystyle-\sum_{l=1}^n \int_x^\infty v_{jl}(t, \lambda)e^{\lambda(b_j - \omega)(p(x) - p(t))}z_l(t)\,dt, & j=\overline{1, k-1},\\
\displaystyle\sum_{l=1}^n \int_\alpha^x v_{jl}(t, \lambda)e^{\lambda(b_j - \omega)(p(x) - p(t))}z_l(t)\,dt, & j=\overline{k, n},
\end{array}
\right.
\end{equation}
where for $j,l = \overline{1, n},$ we put $v_{jl}(x, \lambda) := q_{jl}(x) + r_{jl}(x, \lambda).$ 
Then, for $\lambda \in \overline{\Lambda}\setminus\{0\},$ the operator ${\cal V}_k(\lambda)$ 
is a linear bounded operator in ${\bf BC}_n$ with the following properties.

1) There exists $\lambda_\alpha >0$ such that 
\begin{equation} \label{V^2}
\|{\cal V}^2_k(\lambda)\|_{{\bf BC}_n \to {\bf BC}_n} <\frac12, \quad \lambda \in \overline{\Lambda^\alpha},\quad \Lambda^\alpha := \{ \lambda \in \Lambda \colon |\lambda| > \lambda_\alpha\},
\end{equation}
where we denoted ${\cal V}^\eta_k(\lambda) := \underbrace{{\cal V}_k(\lambda) \ldots {\cal V}_k(\lambda)}_{\eta\text{ times}}.$ 
Moreover, $\lambda_\alpha \to 0$ as $\alpha \to \infty.$

2) For any $\lambda \in \overline{\Lambda^\alpha}$ and $\bf{w} \in {\bf BC}_n,$ the equation
\begin{equation} \label{general eq}
\mathrm{\bf z} = {\bf w} + {\cal V}_k(\lambda) \mathrm{\bf z}
\end{equation}
has a unique solution  $\mathrm{\bf z} \in {\bf BC}_n,$ which  
 satisfies the estimates
 \begin{equation}\label{weak est}
\| \mathrm{\bf z}\|_{ {\bf BC}_n} \le N_\alpha \| \mathrm{\bf w} \|_{ {\bf BC}_n},
\quad \| \mathrm{\bf z} - \mathrm{\bf w}\|_{ {\bf BC}_n} \le N_\alpha \| {\cal V}_k(\lambda){\bf w}\|_{{\bf BC}_n}.
\end{equation}
Here and below, $N_\alpha$ denotes different constants depending only on $\alpha.$

3) Let $\mathrm{\bf w}$ be a mapping of $\lambda$ to ${\bf BC}_n$ that is holomorphic in $\Lambda^\alpha$ and continuous on $\overline{\Lambda^\alpha}.$
Then, the solution of~\eqref{general eq} $\mathrm{\bf z},$ being a mapping of $\lambda$ to ${\bf BC}_n,$ possesses the same properties. 
\end{lemma}
\begin{proof}
For the variable of integration $t$ in~\eqref{operat}, by~\eqref{omega numer}, we have $(p(x) - p(t))\mathrm{Re}\,\big(\lambda (b_j - \omega)\big) \le 0.$ This inequality along with~\eqref{QR} yield that ${\cal V}_k(\lambda)$ is a bounded operator in ${\bf BC}_n$ with the estimate
 \begin{equation} \label{V}
\| {\cal V}_k(\lambda)\|_{{\bf BC}_n \to {\bf BC}_n} \le ne^{2a}(\gamma_\alpha(\lambda) + a).
\end{equation}

 1) Let us consider the operator ${\cal V}^2_k(\lambda).$
For an arbitrary $\mathrm{\bf z} = [z_j(x) ]_{j=1}^n \in {\bf BC}_n,$ put $\mathrm{\bf f} =[f_j(x)]_{j=1}^n := {\cal V}_k^2 \mathrm{\bf z}.$    
By the definition, for $\lambda \in \overline{\Lambda}\setminus\{0\}$ and $j = \overline{1, n},$ we have
$$f_{j}(x) = \sum_{l,m=1}^n \pm \int_{\tau_{j}(x)}^{\sigma_{j}(x)} v_{jl}(t,\lambda) e^{\lambda(b_j - \omega)(p(x) - p(t))} \int_{\tau_{l}(t)}^{\sigma_{l}(t)} v_{lm}(s, \lambda) e^{\lambda(b_l - \omega)(p(t) - p(s))} z_m(s) ds \, dt,$$
where
$\big(\tau_{j}(x), \sigma_{j}(x)\big) :=  \left\{\begin{array}{cc}
(x, \infty), & j=\overline{1, k-1}, \\[1mm]
(\alpha, x), & j=\overline{k, n},
\end{array}\right.$ 
and for each summand, one of two signs $+$ or $-$ should be chosen instead of $\pm.$
Changing the integration order, we obtain
$$
f_j(x) = \sum_{m=1}^n \sum_{l=1}^n \pm \int_\alpha^\infty v_{lm}(s, \lambda)\big(\varkappa_{jl}(s,x,\lambda) +\nu_{jl}(s, x,\lambda)\big)z_m(s) \, ds,$$
and
$$\| \mathrm{\bf f}\|_{{\bf BC}_n} \le \| \mathrm{\bf z}\|_{{\bf BC}_n} \left(\sup_{\substack{j,l=\overline{1, n};\\s,x\ge\alpha}}\big|\varkappa_{jl}(s,x,\lambda)|+\theta_\alpha(\lambda)\right)\sum_{l,m=1}^n  \int_\alpha^\infty |v_{lm}(s,\lambda)| ds.
$$
Put $K_\alpha := \displaystyle \sup_{|\lambda|\ge \varphi(\alpha)} \gamma_\alpha(\lambda) + \| A- D\|_{L[\alpha, \infty)},$ where $\varphi(\alpha)$ is taken from~\eqref{C alpha}. Then, the latter inequality along with~\eqref{QR} and~\eqref{est} yield
\begin{equation} \label{V norm}
\|{\cal V}^2_k(\lambda)\|_{{\bf BC}_n \to {\bf BC}_n} < n^2 e^{2a}  K_\alpha \big(e^{2a}\gamma_\alpha(\lambda) + \theta_{\alpha}(\lambda)\big), \quad \lambda \in \overline{\Lambda}, \; |\lambda| \ge \varphi(\alpha).
\end{equation}
By Lemma~\ref{theta lemma} and condition~\eqref{C con}, we have
$e^{2a}\gamma_\alpha(\lambda) + \theta_{\alpha}(\lambda) \to 0$ as $\lambda \to \infty$ uniformly on $\alpha \ge 0.$
Then, there exists $\lambda_\alpha \ge \varphi(\alpha)$ such that~\eqref{V^2} holds. 

Note that condition~\eqref{C alpha} yields $K_\alpha \to 0$ as $\alpha \to \infty.$  Applying the inequalities $\gamma_{\alpha}(\lambda) \le K_\alpha$ and $\theta_{\alpha}(\lambda) \le a e^{2a}$ to~\eqref{V norm}, for a sufficiently large $\alpha,$ we arrive at
$$\|{\cal V}^2_k(\lambda)\|_{{\bf BC}_n \to {\bf BC}_n} <  n^2 e^{4a} K_\alpha \big(a + K_\alpha\big) < \frac12, \quad \lambda \in \overline{\Lambda}, \; |\lambda| \ge \varphi(\alpha).$$ 
Thus, for a sufficiently large $\alpha,$ we can put $\lambda_\alpha := \varphi(\alpha) \to 0$ as $\alpha \to \infty.$

2) Now, we solve equation~\eqref{general eq} by the method of successive approximations. Consider the series
\begin{equation} \label{sol rep}
\mathrm{\bf z} = \mathrm{\bf w} +\sum_{\eta=0}^\infty {\cal V}^{2\eta+1}_k(\lambda) \mathrm{\bf w} + \sum_{\eta=0}^\infty {\cal V}^{2\eta+2}_k(\lambda) \mathrm{\bf w}, \quad  \lambda \in \overline{\Lambda^\alpha}.
\end{equation}
By~\eqref{V^2}, this series converges in ${\bf BC}_n$ and gives us the unique solution of~\eqref{general eq}. Using also~\eqref{V}, it is easy to obtain estimates~\eqref{weak est}. 

3)  One can show that under our assumptions on $\mathrm{\bf w},$ the mapping ${\cal V}_k(\lambda)\mathrm{\bf w}$ is continuous on $\overline{\Lambda^\alpha}$ and holomorphic in $\Lambda^\alpha$ (for the proof of the holomorphy property, see Proposition~\ref{hol} in Appendix~\ref{appendix B}). By induction, each summand in~\eqref{sol rep} is continuous on $\overline{\Lambda^\alpha}$ and holomorphic in $\Lambda^\alpha.$  The values $\gamma_\alpha(\lambda)$ and $\| \mathrm{\bf w}(\cdot, \lambda)\|_{{\bf BC}_n}$ are continuous functions of $\lambda.$ 
Then, by~\eqref{V^2} and~\eqref{V}, the series in~\eqref{sol rep} converges uniformly on $\lambda$ from every compact subset of $\overline{\Lambda^\alpha}.$ The required statement follows from the fact that the uniform limit of continuous mappings is also continuous  and from Theorem~\ref{lim hol maps} in Appendix~\ref{appendix B}.
\end{proof}

In what follows, we need the following auxiliary statement.
\begin{theorem} \label{theta theorem}
Assume additionally that the elements of the matrix $A - D$ belong to $L_2[\alpha, \infty)$ and $\displaystyle\mathop{\mathrm{essinf}}_{x \ge \alpha}\, \rho(x) > 0.$
Then, for an arbitrary half-line $\Sigma \subset \overline{\Lambda},$ we have $\theta_{\alpha} \in L_{2}(\Sigma).$  
\end{theorem}
Its proof is based on the technique offered in~\cite[Sect.~7]{sav sad}, wherein a similar statement was proved in the case of a finite interval (see also~\cite{sav-calderon}). However, the case of a half-line studied here brings several differences. For this reason, as well as for the sake of completeness, we provide the proof in Appendix~\ref{appendix A}.

\section{Fundamental systems of solutions (FSS)} \label{fss sec}
Let $\mathrm{\bf y}_1, \mathrm{\bf y}_2, \ldots, \mathrm{\bf y}_m$ be a set of $m \ge 1$ solutions of system~\eqref{sys} with some fixed $\lambda \in \mathbb C_0.$
It is convenient to represent this set as the $n \times m$ matrix $Y(x, \lambda)$ whose $k$-th column coincides with $\mathrm{\bf y}_k,$ $k=\overline{1, m}.$ We call the matrix $Y(x, \lambda)$ system of solutions of~\eqref{sys}.

Let a system of solutions $Y(x, \lambda)$ consist of $m=n$ columns. If its columns are linearly independent on $x \in [0, \infty),$
$Y(x, \lambda)$ is called FSS of~\eqref{sys}.
It is known that criterion of FSS is $\det Y(x, \lambda) \ne 0$ for some $x \ge 0.$ 

In this section, we obtain fundamental systems of solutions $Y_\alpha(x, \lambda)$ depending on the parameter $\alpha \ge 0.$ They are constructed individually for $\lambda$ belonging to special sectors. Consider the set of the lines $\mathrm{Re}\, (\lambda b_j) = \mathrm{Re}\, (\lambda b_l),$ where $j,l=\overline{1,n}$ and $b_j \ne b_l$ (see~\cite{sav,beals}). These lines split the plane $\lambda \in \mathbb C$ into sectors of the form $\Gamma_\kappa = \{ \lambda \in \mathbb C_0 \colon \beta_{\kappa-1} < \arg \lambda < \beta_{\kappa}\}.$ For each such sector $\Gamma_\kappa,$ one can renumber $\{ b_j\}_{j=1}^n$ so that~\eqref{numer1} holds. Obviously, there is a finite number $J$ of the sectors $\Gamma_\kappa,$ $\kappa=\overline{1, J},$ they are non-overlapping, and $\mathop{\cup}_{\kappa = 1}^J \overline{\Gamma_\kappa} = \mathbb C.$ 

Let us consider $\lambda \in \overline{\Gamma_\kappa}$
with fixed $\kappa \in \overline{1, J}$ and proceed to a numeration satisfying~\eqref{numer1}.  
Note that an arbitrary numeration of $\{ b_j \}_{j=1}^n$ can be achieved by permutations of rows and columns of the objects in~\eqref{sys}. In fact, the mentioned permutations do not influence the formulations of Theorems~\ref{fss} and~\ref{complement}.

\begin{theorem} \label{fss}
For any $\alpha \ge 0,$ there exists $\lambda_{\alpha}>0$ such that for
$\lambda \in \overline{\Gamma^{\alpha}_\kappa},$ ${\Gamma}^{\alpha}_\kappa := \{ \lambda \in \Gamma_\kappa \colon |\lambda| > \lambda_\alpha\},$ there exists a 
 FSS of~\eqref{sys} $Y_\alpha(x, \lambda) = [y_{jk}(x, \lambda)]_{j,k=1}^n$ with the following properties.

\begin{enumerate}
\item  
For $j,k=\overline{1,n}$ and $\lambda \in \overline{\Gamma^\alpha_\kappa},$ uniformly on $x \ge \alpha,$ we have
\begin{equation} \label{asymptotics}
y_{jk}(x, \lambda) = e^{\lambda b_k (p(x) - p(\alpha))}\big(m_{jk}(x) +s_{jk}(x, \lambda)\big) , \quad s_{jk}(x, \lambda) = o(1),  \quad  \lambda \to \infty,
\end{equation}
where $m_{jk}(x)$ are the elements of the matrix $M_\alpha$ determined by~\eqref{M init value}.

\item  For a fixed $x  \ge 0,$ the functions $y_{jk}(x, \lambda),$ $j,k=\overline{1, n,}$ are continuous on $\overline{\Gamma_\kappa^\alpha}$ and  analytic in $ \Gamma_\kappa^{\alpha}.$ 

\item $
y_{jk}(\alpha, \lambda) = \delta_{jk}$ for $k = \overline{1, n}$ and  $j = \overline{k, n},$ where $\delta_{jk}$ is the Kronecker delta. 

\end{enumerate}
Moreover,  $\lambda_\alpha \to 0$ as $\alpha \to \infty.$
\end{theorem}
In the proof, we need the following proposition, which is a consequence of Proposition~2 and Corollary~1 in~\cite{sav}.

\begin{prop} \label{finite interval prop}
For $\alpha > 0$ and $\lambda \in \mathbb C_0,$ consider the initial-value problem for system~\eqref{sys} on the finite segment $x \in [0, \alpha]$ under the initial condition $\mathrm{\bf y}(\alpha, \lambda) = [w_j(\lambda)]_{j=1}^n.$ Then, 
the initial-value problem has a unique solution $\mathrm{\bf y}(x, \lambda) = [y_j(x, \lambda)]_{j=1}^n$ with the components $y_{j}(\cdot, \lambda) \in AC[0, \alpha].$ Moreover, if the functions $w_l(\lambda),$ $l = \overline{1, n},$ are analytic in (continuous on) $\Lambda \subseteq \mathbb C_0,$ then for a fixed $x \in [0, \alpha],$  the functions $y_j(x, \lambda),$ $j = \overline{1, n},$ are also analytic in (continuous on) $\Lambda.$
\end{prop}

\begin{proof}[Proof of Theorem~\ref{fss}] First, we construct the needed solutions of~\eqref{sys}  for $x \ge \alpha,$ then we extend them on $[0, \alpha]$ applying Proposition~\ref{finite interval prop}. 
For $k = \overline{1, n}$ and $x \ge \alpha,$ let us represent the $k$-th column of the matrix function $Y_\alpha(x, \lambda)$ as
\begin{equation}\label{y_k}
[y_{jk}(x, \lambda)]_{j=1}^n = M_\alpha(x) \mathrm{\bf z}_k(x, \lambda) e^{\lambda b_k (p(x) - p(\alpha))}, \quad \mathrm{\bf z}_k(x, \lambda) = [z_{jk}(x, \lambda)]_{j=\overline{1,n}}.
\end{equation}
Substituting $\mathrm{\bf y} =[y_{jk}(x, \lambda)]_{j=1}^n$ into~\eqref{sys}, we arrive at the equation  
 \begin{equation} \label{z sys}
\mathrm{\bf z}'_k(x, \lambda) = \lambda r(x) (B \mathrm{\bf z}_k(x, \lambda) - b_k \mathrm{\bf z}_k(x, \lambda)) + (Q_\alpha(x) + R_\alpha(x, \lambda)) \mathrm{\bf z}_k(x, \lambda), \quad x \ge \alpha.
\end{equation}
Now, we construct a solution of~\eqref{z sys} $\mathrm{\bf z}_k \in {\bf BC}_n$ satisfying the conditions
\begin{equation}\label{init}
\left\{\begin{array}{cc}
\displaystyle\lim_{x \to \infty}z_{jk}(x, \lambda) = 0, & j=\overline{1, k - 1}, \\[2mm]
z_{jk}(\alpha, \lambda) = \delta_{jk}, & j =\overline{ k, n}.
\end{array}\right. 
\end{equation}
Integrating~\eqref{z sys} under conditions~\eqref{init}, for each fixed $k\in\overline{1, n},$  we obtain the equation
\begin{equation} \label{int sys}
\mathrm{\bf z}_k = \mathrm{\bf e}_k + {\cal V}_k(\lambda) \mathrm{\bf z}_k, \quad \mathrm{\bf e}_k := [\delta_{jk}]_{j=1}^n,
\end{equation}
where the operator ${\cal V}_k(\lambda)$ is determined in Lemma~\ref{main lemma} with $\omega := b_k$ and $\Lambda := \Gamma_\kappa.$ 
We can apply Lemma~\ref{main lemma}, since~\eqref{numer1} yields~\eqref{omega numer}. Then, for $\lambda \in \overline{\Gamma_\kappa^\alpha},$ equation~\eqref{int sys} has a unique solution $\mathrm{\bf z}_k \in {\bf BC}_n,$ and 
\begin{equation} \label{z_k estimate}
\| \mathrm{\bf z}_k - \mathrm{\bf e}_k \|_{{\bf BC}_n} \le N_\alpha \|\mathrm{\bf z}^1_k\|_{{\bf BC}_n}, \quad \mathrm{\bf z}^1_k := {\cal V}_k(\lambda) \mathrm{\bf e}_k.
\end{equation}
Introduce operators ${\cal R}_k(\lambda)$ and ${\cal Q}_k(\lambda)$ analogously to ${\cal V}_k(\lambda)$ in~\eqref{operat}, replacing each component $v_{jl}(t, \lambda)$ by $r_{jl}(t, \lambda)$ and $q_{jl}(t),$ respectively.  
By the same way as~\eqref{V}, we obtain that 
$\| {\cal R}_k(\lambda)\|_{{\bf BC}_n \to {\bf BC}_n} \le n e^{2a} \gamma_\alpha(\lambda).$
We also have
${\cal Q}_k \mathrm{\bf e}_{k} = [\nu_{jk}(\alpha, x,\lambda)]_{j=1}^n,$ and, consequently,
$\|  {\cal Q}_k \mathrm{\bf e}_{k}\|_{{\bf BC}_n} \le \theta_{\alpha}(\lambda).$
Then, since ${\cal V}_k(\lambda) = {\cal R}_k(\lambda)+{\cal Q}_k(\lambda),$ we obtain
\begin{equation*} 
\| \mathrm{\bf z}^1_{k} \|_{{\bf BC}_n} \le \| {\cal R}_k(\lambda)\mathrm{\bf e}_k\|_{{\bf BC}_n} + \|{\cal Q}_k \mathrm{\bf e}_{k}\|_{{\bf BC}_n} \le n e^{2a} \gamma_\alpha(\lambda) + \theta_{\alpha}(\lambda).
\end{equation*}
Using~\eqref{z_k estimate} and the latter inequality, we arrive at
\begin{equation} \label{z_k^0}
\| \mathrm{\bf z}_k - \mathrm{\bf e}_k\|_{{\bf BC}_n} \le N_\alpha ( \gamma_\alpha(\lambda) +\theta_{\alpha}(\lambda)), \quad \lambda \in \overline{\Gamma_\kappa^\alpha}.
\end{equation}

Constructing the matrix $Y_\alpha(x, \lambda) = [y_{jk}(x, \lambda)]_{j,k=1}^n$ with the components get by~\eqref{y_k}, we obtain a system of solutions of~\eqref{sys} for $x \ge \alpha,$ possessing the needed properties 2 and 3. Property~3 yields that $Y_\alpha(x, \lambda)$ is a FSS.
 From~\eqref{M alpha},~\eqref{y_k} and~\eqref{z_k^0} it follows that formulae~\eqref{asymptotics} hold with 
\begin{equation}
|s_{jk}(x, \lambda)| \le N_\alpha (\gamma_\alpha(\lambda) +\theta_{\alpha}(\lambda)), \quad j,k=\overline{1, n}, \; x \ge \alpha, 
\label{s est}
\end{equation}
wherein, by Lemma~\ref{theta lemma} and condition~\eqref{C con}, the right side tends to $0$ as $\lambda \to \infty.$
Thus, property~1 is also established.

It remains to extend each column vector $[y_{jk}(x, \lambda)]_{j=1}^n$ on $x \in [0, \alpha]$ by the solution of system~\eqref{sys} under the initial condition $\mathrm{\bf y}(\alpha, \lambda)=[y_{jk}(\alpha, \lambda)]_{j=1}^n.$ By virtue of Proposition~\ref{finite interval prop}, the needed solution exists and keeps property~2.
\end{proof}

Further, we give a complement to Theorem~\ref{fss} under additional assumptions on the matrices $A(x)$ and $C(x, \lambda),$ that concerns properties of the residual members $s_{jk}(x, \lambda)$ in~\eqref{asymptotics}.
\begin{theorem}
Assume that the elements of the matrix $A - D$ belong to $L_2[\alpha, \infty),$ 
$\displaystyle\mathop{\mathrm{essinf}}_{x \ge \alpha}\, \rho(x) > 0,$ and $\gamma_\alpha(\lambda)= O(\lambda^{-1})$ for $|\lambda|\ge \lambda_\alpha.$ Then, for an arbitrary half-line $\Sigma \subset \overline{\Gamma_\kappa^\alpha}$ and the values $s_{jk}(x, \lambda)$ in~\eqref{asymptotics}, we have
$$\sup_{x \ge \alpha} |s_{jk}(x, \lambda)| \in L_2(\Sigma), \quad j,k = \overline{1, n}.$$
\label{complement}
\end{theorem} 
The statement of the theorem is a simple corollary of inequality~\eqref{s est} and Theorem~\ref{theta theorem}.


\section{Solutions analytic in large sectors} \label{large sec}
In this section, we restrict ourselves to the case when $\{b_j\}_{j=1}^n$ is the set of all $n$-th roots of unity, i.e. up to a numeration,  $b_j = \exp(\frac{2\pi i j}{n}),$ $j=\overline{1,n}.$  
Remind that exactly this case realizes after reducing $n$-th order equation~\eqref{n-th order} to system~\eqref{sys}.

If $n > 2,$ the sectors $\Gamma_\kappa$ introduced in Section~\ref{fss sec} have the form
\begin{equation} \label{sectors}
\Gamma_\kappa = \left\{ \lambda \in {\mathbb C}_0 \colon \frac{\pi(\kappa-1)}{n} < \arg \lambda < \frac{\pi \kappa}{n} \right\}, \quad \kappa = \overline{1, 2n}.
\end{equation}
If $n=2,$ then we have two sectors: the right half-plane $\Gamma_1 = \{ \lambda \in \mathbb C \colon \mathrm{Re}\, \lambda > 0\}$ and the left half-plane $\Gamma_2 = \{ \lambda \in \mathbb C \colon \mathrm{Re}\, \lambda < 0\}.$ Further, we consider the case $n > 2.$ The results of this section will be valid for $n = 2$ as well, if one takes formula~\eqref{sectors} as the definition of the sectors $\Gamma_\kappa,$ $\kappa=\overline{1,4}.$ However, in this case, Theorem~\ref{neighboring} is weaker than Theorem~\ref{fss}.

For $m \in \overline{2, n},$ consider a large sector including two neigboring sectors $\Gamma_\kappa:$ 
$$\Omega_m = \left\{ \lambda \in \mathbb C \colon \arg\lambda \in \Big([(-1)^{m-1} - 1]\frac{\pi}{2n}; \ [(-1)^{m-1}+3]\frac{\pi}{2n}\Big)\right\}.$$
It is clear that 
$\overline{\Omega_m} = \overline{\Gamma_1} \cup \overline{\Gamma_\sigma},$ 
where $\sigma = 2n$ if $m$ is even and $\sigma = 2$ if $m$ is odd.
In what follows, we fix the numeration of $\{ b_j\}_{j=1}^n$ for which~\eqref{numer1} holds in $\kappa = 1:$ 
\begin{equation} \label{b odd even}
b_{2s+1} =  e^{\frac{2 \pi i s}{n}}, \; s = \overline{0, \left\lfloor\frac{n-1}{2}\right\rfloor}; \quad b_{2p} =  e^{-\frac{2 \pi i p}{n}}, \; p = \overline{1, \left\lfloor\frac{n}{2}\right\rfloor}.
\end{equation}
For this fixed numeration and $\lambda$ from the neighboring sector $\Gamma_\sigma,$ some inequalities in~\eqref{numer1} become incorrect. 
One can see that
\begin{equation} \label{omega2}
\mathrm{Re} \,(\lambda b_j) \ge \mathrm{Re} \, (\lambda b_{m+1}) \ge \mathrm{Re} \, (\lambda b_{l}), \quad j=\overline{1, m-1}, \; l = \overline{m, n}, \quad \lambda \in \overline{\Gamma_\sigma},
\end{equation}
where and below we put $b_{n+1} := b_n$ in the case $m=n.$ 

Now, we obtain non-fundamental system of solutions of~\eqref{sys} with analytic dependence on the spectral parameter from the large sector $\Omega_m.$
\begin{theorem} \label{neighboring}
Let $\alpha\ge 0,$ $m \in \overline{2, n},$ and $\Omega_m^\alpha := \{ \lambda \in \Omega_m \colon |\lambda|>\lambda_\alpha \}.$ For $\lambda \in \overline{\Omega_m^\alpha},$ there exists a system of solutions of~\eqref{sys} $U_\alpha(x, \lambda) = [u_{jk}(x, \lambda)]_{\substack{j=\overline{1,n}, \\k = \overline{m, n}}}$ possessing the following properties.
\begin{enumerate}
\item
For $j=\overline{1, n}$ and $k=\overline{m, n},$ uniformly on $x \ge\alpha,$ we have
\begin{equation*} \label{Z1}
u_{jk}(x, \lambda) = \left\{\begin{array}{cc}
O\left(e^{\lambda b_m (p(x)-p(\alpha))}\right), & \lambda \in \overline{\Gamma^\alpha_1}, \\
O\left(e^{\lambda b_{m+1} (p(x)-p(\alpha))}\right), & \lambda \in \overline{\Gamma_\sigma^\alpha}.
\end{array}\right.
\end{equation*}

\item For a fixed $x \ge 0,$ the functions $u_{jk}(x, \lambda),$ $j=\overline{1, n},$ $k=\overline{m, n},$ are continuous on $\overline{\Omega^\alpha_m}$ and analytic in $\Omega^\alpha_m.$

\item $
u_{jk}(\alpha, \lambda) = \delta_{jk},$ $ j,k=\overline{m, n}.
 $
\end{enumerate}
\end{theorem}
 In the proof, we need a lemma.
 \begin{lemma}
Consider the number $\omega = b_m \exp\big(\frac{(-1)^m \pi i}{n}\big)$ and the sector
\begin{equation} \label{mid Lambda}
\Lambda = \left\{\lambda \in \mathbb C_0 \colon \arg \lambda \in \Big(\frac{(-1)^{m-1}\pi}{2n}; \frac{(-1)^{m-1}\pi}{2n} +\frac{\pi}{n}\Big)\right\},
\end{equation}
see Figure~\ref{picture}.
The following inequalities hold:
\begin{equation}\label{numer2}
\mathrm{Re}\,(\lambda b_j) \ge \mathrm{Re}\,(\lambda \omega) \ge\mathrm{Re}\,(\lambda b_l), \quad j=\overline{1, m-1},\;l=\overline{m,n}, \quad \lambda \in \overline{\Lambda}.
\end{equation}  
 \end{lemma}
 \begin{figure}
\begin{minipage}{0.45\textwidth}
\begin{center}
 {$m$ is odd}
%
%
%
%
%
%
%
%
\includegraphics{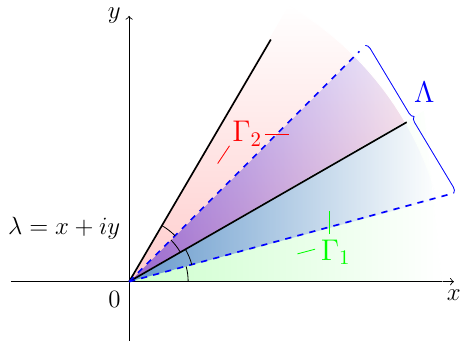}
\end{center}
\end{minipage} 
\hspace{1cm}
\begin{minipage}{0.45\textwidth}
\begin{center}
 {$m$ is even}
%
%
%
%
%
%
%
%
\includegraphics{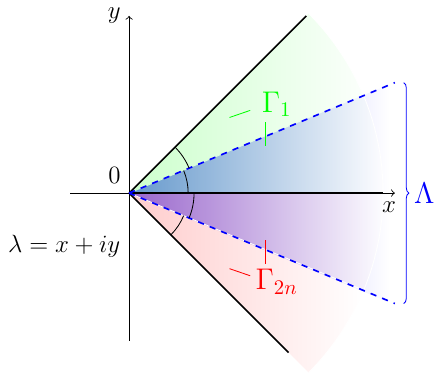}
\end{center}
\end{minipage}
  \caption{The sectors $\Gamma_1,$ $\Gamma_\sigma,$ and $\Lambda.$}
  \label{picture}
\end{figure}

 \begin{proof}
It is sufficient to consider only $\lambda$ such that $|\lambda|=1.$ Then, $\lambda \in \overline{\Lambda}$ if and only if $\lambda = e^{i \tau}$ for $\tau \in \Big[\frac{(-1)^{m-1}\pi}{2n}; \frac{(-1)^{m-1}\pi}{2n} +\frac{\pi}{n}\Big].$ For any fixed $z \in \mathbb C_0,$ ${\mathrm Re}\,(\lambda z)$ runs over a closed segment as $\lambda$ runs over $\overline{\Lambda}$ and $|\lambda|=1.$ Let us find these segments  for $z \in \{ \omega, b_1, \ldots, b_n\}.$ 

For definiteness, we assume that $m$ is odd; the case of even $m$ is proceeded analogously. By~\eqref{b odd even}, for $s=\overline{0, \left\lfloor\frac{n-1}{2}\right\rfloor},$ we have
$${\mathrm Re}\, (\lambda b_{2s+1}) = {\mathrm Re}\, e^{i\big(\tau + \frac{2\pi s}{n}\big)} = \cos \Big(\tau + \frac{2\pi s}{n}\Big),$$ 
$$\tau + \frac{2\pi s}{n} \in \big[(4s+1)\zeta; \, (4s+3)\zeta \big] \subset \left[\zeta; \pi + \zeta\right], \quad \zeta := \frac{\pi}{2n}.$$
Since the cosine function decreases on $[\zeta;\, \pi]$ and increases on $[\pi ; \, \pi +\zeta]$ (the latter segment appears only for $s = \frac{n-1}{2}$ if $n$ is odd), we have
\begin{equation} \label{2s+1}
{\mathrm Re}\, (\lambda b_{2s+1}) \in \left[\cos \big(\min\{(4s+3)\zeta, \pi\}\big);\, \cos \big((4s+1)\zeta\big)\right], \quad s = \overline{0, \left\lfloor\frac{n-1}{2}\right\rfloor}. 
\end{equation}

Analogously, we obtain
\begin{equation} \label{2s}
{\mathrm Re}\, (\lambda b_{2p}) \in \left[\cos \big((4p-1)\zeta\big);\, \cos \big((4p-3)\zeta\big)\right], \quad p = \overline{1, \left\lfloor\frac{n}{2}\right\rfloor},
\end{equation}  
\begin{equation} \label{omega segments}
{\mathrm Re}\, (\lambda \omega) \in \left[\cos \big((4k+1)\zeta\big);\, \cos \big(\max\{(4k-1)\zeta, 0\}\big)\right], \quad k = \frac{m-1}{2}.
\end{equation}  
One can see that the segments in~\eqref{omega segments} can intersect the segments in~\eqref{2s+1}--\eqref{2s} only at the boundary points. Then, 
$${\mathrm Re}\, (\lambda \omega) \ge {\mathrm Re}\, (\lambda b_{2s+1}) \Leftrightarrow \cos \big((4s+1)\zeta\big)\pi \le \cos \big((4k+1)\zeta\big) \Leftrightarrow s \ge k \Leftrightarrow 2s+1 \ge m.$$  
$${\mathrm Re}\, (\lambda \omega) \ge {\mathrm Re}\, (\lambda b_{2p}) \Leftrightarrow \cos \big((4p-3)\zeta\big)\pi \le \cos \big((4k+1)\zeta\big) \Leftrightarrow p -1\ge k \Leftrightarrow 2p \ge m.$$ 
Thus, we have proved~\eqref{numer2} for all possible values $j=2s+1$ or $j=2p.$
 \end{proof}
\begin{proof}[Proof of Theorem~\ref{neighboring}.]
As in the proof of Theorem~\ref{fss}, applying substitutions~\eqref{y_k} for each fixed $k=\overline{m,n},$
we obtain system~\eqref{z sys} with respect to the column vector $\mathrm{\bf z}_k = [z_{jk}(x, \lambda)]_{j=1}^n.$  Integrating it under the conditions
\begin{equation}\label{bc2}
\left\{\begin{array}{cc}
\displaystyle\lim_{x \to \infty} z_{jk}(x, \lambda) = 0, & j=\overline{1, m-1},\\
z_{jk}(\alpha, \lambda) = \delta_{jk}, & j=\overline{m, n},
\end{array} \right.
\end{equation}
 we arrive at the system of integral equations
 \begin{equation} \label{m sys}
z_{jk}(x, \lambda) = \left\{\begin{array}{cc}
 \displaystyle  -\sum_{l=1}^n \int_x^\infty v_{jl}(t, \lambda) e^{\lambda(b_j - b_k)(p(x)-p(t))} z_{lk}(t, \lambda) \, dt, & j=\overline{1, m-1}, \\[2mm]
\displaystyle\delta_{jk} + \sum_{l=1}^n \int_\alpha^x v_{jl}(t, \lambda) e^{\lambda(b_j - b_k)(p(x)-p(t))} z_{lk}(t, \lambda) \, dt,  & j=\overline{m, n}.
\end{array}\right.
\end{equation}
Let us study solvability of this system in the class of vector functions $\mathrm{\bf z}_k = [z_{jk}(x, \lambda)]_{j=1}^n,$ $x \ge \alpha,$ whose components belong to $AC[\alpha, T]$ for any $T > \alpha.$

First, we consider $\lambda \in \overline{\Lambda},$ where $\Lambda$ is given in~\eqref{mid Lambda}, see Figure~\ref{picture}.
Substituting $\mathrm{\bf z}_k = e^{\lambda(\omega - b_k)(p(x) - p(\alpha))}\mathrm{\bf \tilde z}_k$ into~\eqref{m sys}, we obtain an equation with respect to $\mathrm{\bf \tilde z}_k:$
\begin{equation} \label{eqq}
\mathrm{\bf \tilde z}_k = \mathrm{\bf w}_k + {\cal V}_m(\lambda) \mathrm{\bf \tilde z}_k, \quad \mathrm{\bf w}_k := e^{\lambda(b_k - \omega)(p(x) - p(\alpha))} \mathrm{\bf e}_k,
\end{equation}
where ${\cal V}_m(\lambda)$ is determined in Lemma~\ref{main lemma}, while $\mathrm{\bf e}_k$ is given in~\eqref{int sys}.
Inequalities~\eqref{numer2} yield that $\| \mathrm{\bf w}_k \|_{\bf BC_n} \le 1$ and assure applicability of Lemma~\ref{main lemma}.
 By virtue of this lemma, for $\lambda \in \overline{\Lambda^\alpha},$ there exists a unique solution of~\eqref{eqq} $\mathrm{\bf \tilde z}_k \in {\bf BC}_n.$  
Consequently, for  $\lambda \in \overline{\Lambda^\alpha},$ there exists a solution of~\eqref{m sys} $\mathrm{\bf z}^0_k(\lambda) = e^{\lambda(\omega - b_k)(p(x) - p(\alpha))}\mathrm{\bf \tilde z}_k$ such that
\begin{equation} \label{z_k ineq}
\mathrm{\bf z}^0_k=[z^0_{jk}(x,\lambda)]_{j=1}^n, \quad |z^0_{jk}(x,\lambda)|\le N_\alpha \left|e^{\lambda(\omega - b_k)(p(x) - p(\alpha))}\right|, \quad x \ge \alpha, \; \lambda \in \overline{\Lambda^\alpha}, \quad j = \overline{1, n},
\end{equation}
with some $N_\alpha >0.$ 
Moreover, for each fixed $x \ge \alpha,$  the components  $z^0_{jk}(x, \lambda),$  $j=\overline{1,n},$ are analytic in $\Lambda^\alpha$ and continuous on $\overline{\Lambda^\alpha}.$ Note that the solution of~\eqref{m sys} satisfying~\eqref{z_k ineq} is {\it unique}, otherwise equation~\eqref{eqq} would have at least two solutions from ${\bf BC}_n.$ 

Considering $\Gamma_1$ instead of $\Lambda$  and $b_m$ instead of $\omega,$ 
we can apply the same reasoning as in the previous paragraph, because, by the numeration of $\{ b_j\}_{j=1}^n,$ inequalities~\eqref{numer2} hold with $\omega = b_m$ and $\Lambda = \Gamma_1.$
Then, for $\lambda \in \overline{\Gamma_1^\alpha},$ there exists a solution of~\eqref{m sys}  $\mathrm{\bf z}^1_k(\lambda)=[z^1_{jk}(x,\lambda)]_{j=1}^n$ satisfying 
\begin{equation}
\label{a1}
|z^1_{jk}(x, \lambda)| \le N_\alpha \big|e^{\lambda(b_m - b_k)(p(x) - p(\alpha))}\big|, \quad x \ge \alpha, \; \lambda \in \overline{\Gamma_1^\alpha}, \quad j=\overline{1, n}.
\end{equation}
For each fixed $x \ge \alpha,$ its components are continuous on $\overline{\Gamma_1^\alpha}$ and analytic in $\Gamma_1^\alpha.$ 

Put $\mathrm H_1 = \Lambda^\alpha \cap \Gamma_1^\alpha$ and $\Xi_1 = \Lambda^\alpha \cup \Gamma_1^\alpha.$ Now, we prove that $\mathrm{\bf z}^0_k(\lambda)= \mathrm{\bf z}^1_k(\lambda)$ for $x\ge\alpha$ and $\lambda \in \overline{\mathrm H_1}.$ By~\eqref{numer2}, we have  
$$\big|e^{\lambda(b_m - b_k)(p(x) - p(\alpha))}\big|\le \left|e^{\lambda(\omega - b_k)(p(x) - p(\alpha))}\right|, \quad x\ge\alpha, \; \lambda \in \overline{\mathrm H_1}.$$ Then, in the overlapping part of two sectors $\lambda \in \overline{\mathrm H_1},$ the solution of~\eqref{m sys} $\mathrm{\bf z}^1_k(\lambda)$ satisfies~\eqref{z_k ineq}. However, for $\lambda \in \overline{\Lambda^\alpha},$ such solution is unique. By this reason, $\mathrm{\bf z}^0_k(\lambda)$ and $\mathrm{\bf z}^1_k(\lambda)$ coincide for $\lambda \in \overline{\mathrm H_1}.$
 This means that for $\lambda \in \overline{\Xi_1},$ we can consider a unified solution $\mathrm{\bf z}_k(\lambda)=[z_{jk}(x,\lambda)]_{j=1}^n$ putting 
 $\mathrm{\bf z}_k(\lambda) = \mathrm{\bf z}^0_k(\lambda)$ for $\lambda \in \overline{\Lambda^\alpha}$ and $\mathrm{\bf z}_k(\lambda) = \mathrm{\bf z}^1_k(\lambda)$ for $\lambda \in \overline{\Gamma_1^\alpha}\setminus\overline{\mathrm H_1}.$
  For each fixed $x \ge \alpha,$ the components $z_{jk}(x,\lambda),$ $j=\overline{1,n,}$ are continuous on $\overline{\Xi_1}$ and analytic in $\Xi_1,$ being analytic continuations.

Analogously, using~\eqref{omega2}, for $\lambda \in \overline{\Gamma_\sigma^\alpha},$ we construct a solution of~\eqref{m sys} $\mathrm{\bf z}^\sigma_k=[z^\sigma_{jk}(x,\lambda)]_{j=1}^n$  satisfying
\begin{equation}\label{a2}
|z^\sigma_{jk}(x, \lambda)| \le N_\alpha \big|e^{\lambda(b_{m+1} - b_k)(p(x) - p(\alpha))}\big|, \quad x \ge \alpha,\; \lambda \in \overline{\Gamma_\sigma^\alpha}, \quad j=\overline{1, n}.
\end{equation}
For each fixed $x \ge \alpha,$ its components are continuous on $\overline{\Gamma_\sigma^\alpha}$ and analytic in $\Gamma_\sigma^\alpha.$ Put $\mathrm H_\sigma = \Lambda^\alpha \cap \Gamma_\sigma^\alpha$ and $\Xi_\sigma = \Lambda^\alpha \cup \Gamma_\sigma^\alpha.$ Obviously, $\Xi_1 \cup \Xi_\sigma = \Omega_m^\alpha.$
Applying the reasoning similar to the previous paragraph, we obtain that $\mathrm{\bf z}^\sigma_k(\lambda) = \mathrm{\bf z}^0_k(\lambda)$  for $\lambda \in \overline{\mathrm H_\sigma}.$
Thus, the solution $\mathrm{\bf z}^\sigma_k(\lambda)$ extends  $\mathrm{\bf z}_k$  from the overlapping part $\lambda \in\overline{\mathrm H_\sigma}$ on $\lambda \in \overline{\Xi_\sigma}.$ 
 Putting $\mathrm{\bf z}_k(\lambda) = \mathrm{\bf z}^\sigma_k(\lambda)$ for $\lambda \in \overline{\Gamma_\sigma^\alpha}\setminus\overline{\mathrm H_\sigma},$ we obtain a unified solution $\mathrm{\bf z}_k=[z_{jk}(x,\lambda)]_{j=1}^n,$ whose components  are  continuous on $\overline{\Omega_m^\alpha}$ and analytic in $\Omega_m^\alpha$ for each fixed $x \ge \alpha.$ 

Consider $$u_{jk}(x, \lambda) = \exp\left(\int_\alpha^x a_{jj}(t)\,dt + \lambda b_k (p(x) - p(\alpha))\right)z_{jk}(x, \lambda), \quad j=\overline{1,n},\;k = \overline{m,n}.$$ 
Then, the matrix $U_\alpha(x, \lambda) = [u_{jk}(x, \lambda)]_{\substack{j=\overline{1,n},\\ k=\overline{m, n}}}$ is a system of solutions of~\eqref{sys} for  $x \ge \alpha.$ 
Using~\eqref{bc2}, \eqref{a1}, and~\eqref{a2}, it is easy to obtain properties~1 and~3. Analyticity and continuity of $z_{jk}$ yield property~2 for $x \ge \alpha.$
 It remains to extend  each column vector $[u_{jk}(x, \lambda)]_{j=1}^n$ on $x \in [0, \alpha]$ by the solution of system~\eqref{sys} under the initial condition $\mathrm{\bf y}(\alpha, \lambda)=[u_{jk}(\alpha, \lambda)]_{j=1}^n$ and apply Proposition~\ref{finite interval prop}.
\end{proof}

In Theorem~\ref{neighboring}, we obtained the system of solutions
 $U_\alpha(x, \lambda)$ consisting of $n-m+1$ vector functions, which is non-fundamental. However, for $\lambda \in \overline{\Gamma^\alpha_1}$ or $\lambda \in \overline{\Gamma^\alpha_\sigma},$ it can be supplemented with the first $m-1$ columns of the system $Y_\alpha(x, \lambda)$ constructed in Theorem~\ref{fss}. Then, by properties~3 of Theorems~\ref{fss} and~\ref{neighboring}, we obtain a FSS.

\begin{remark}
One can consider system~\eqref{sys} on the line $x \in \mathbb R$ under the conditions similar to~I--III, substituting $[0, \infty)$ by $\mathbb R$ and assuming that $\rho \in L(-T, T)$ for any $T > 0.$  
For finite $\alpha\in \mathbb R,$ the results of Theorems~\ref{fss}--\ref{neighboring} are carried to the case of the line, since we can construct the solutions with the needed properties for $x \ge \alpha$ and extend them on $x \in [-T, \alpha]$ for any $T > 0.$ 
Note that the line has other specifics concerning the statement of inverse spectral problems.  Inverse scattering problems considered in~\cite{beals, nova} require the solutions determined by the behavior at $x \to \infty,$ that differ from the solutions constructed here.
\end{remark}

 
 \section{Application to second-order equations with distribution potentials} \label{n=2}

In this section, we consider application of our results to the following second-order equation with respect to a function $u:$
\begin{equation}\label{pencil}
-u''(x) + q(x) u(x) + z p_0(x) u(x) = z^2 u(x), \quad x \ge 0,
\end{equation}
where $z \in \mathbb C$ is the spectral parameter. 
We assume that for some $\sigma \in L \cap L_2[0, \infty),$ $q = \sigma'$ in the sense of distributions; then, $q$ is a distribution potential.
Let also $p_0\in L \cap L_2[0, \infty).$ In the particular case $p_0 = 0$ equality~\eqref{pencil} turns into the Sturm--Liouville equation with a distribution potential, see~\cite{S-L regularization}.
Further, we reduce~\eqref{pencil} to first-order system~\eqref{sys} with $n=2.$ Then, we apply Theorems~\ref{fss} and~\ref{complement} to obtain FSS of~\eqref{pencil}. 
For equation~\eqref{pencil} with distribution potential $q,$ construction of other solutions was considered in~\cite{manko, pronska}. 

First, we apply a regularization approach offered in~\cite{S-L regularization}.
 Introduce the quasi-derivative
\begin{equation} \label{quasi-derivative}
u^{[1]} = u' - \sigma u.
\end{equation}
Imposing the restrictions $u, u^{[1]} \in AC[0, T)$ for any $T>0,$
 we can rewrite~\eqref{pencil} in the regularized form
\begin{equation} \label{regularized}
-(u^{[1]})' - \sigma u^{[1]} - \sigma^2 u + z p_0(x) u = z^2 u, \quad x \ge 0, \; z \in \mathbb C,
\end{equation}
where the left side belongs to $L[0, T)$ for any $T >0.$
Equalities~\eqref{quasi-derivative} and~\eqref{regularized} should be treated together as a first-order system with respect to the functions $u$ and $u^{[1]}.$

Consider the spectral parameter $\lambda = z i \in \mathbb C_0$ and the vector-function $\mathrm{\bf v} = [v_j]_{j=1}^2,$ where $v_1 = u$ and $v_2 = \frac{u^{[1]}}{\lambda}.$  Then,~\eqref{quasi-derivative} and~\eqref{regularized} are equivalent to the system
\begin{equation} \label{sys prev}
\mathrm{\bf v}' = \lambda \begin{pmatrix}
0 & 1 \\
1 & 0
\end{pmatrix} \mathrm{\bf v} + \begin{pmatrix}
\sigma & 0 \\
-p_0i & -\sigma
\end{pmatrix} \mathrm{\bf v} -
\begin{pmatrix}
0 & 0 \\
\frac{\sigma^2}{\lambda} & 0
\end{pmatrix} \mathrm{\bf v}.
\end{equation}
It remains to bring this system to form~\eqref{sys}. For this purpose, consider the matrices
$$\Theta = \begin{pmatrix}
1 & 1 \\
1 & -1
\end{pmatrix}, \quad  \Theta^{-1} = \frac{1}{2}\begin{pmatrix}
1 & 1\\
1 & -1
\end{pmatrix}.$$
Making substitution $\mathrm{\bf y} = \Theta^{-1} \mathrm{\bf v},$ for $\lambda \in \mathbb C_0,$ we arrive at the equivalent system
\begin{equation} \label{sys coeff} \mathrm{\bf y}' =
\lambda \begin{pmatrix}
1 & 0 \\
0 & -1
\end{pmatrix}\mathrm{\bf y} +
\begin{pmatrix}
\frac{-ip_0}{2} & \sigma - \frac{ip_0}{2} \\
 \sigma+\frac{ip_0}{2} & \frac{ip_0}{2}
\end{pmatrix}\mathrm{\bf y}+ \frac{\sigma^2}{2\lambda}\begin{pmatrix}
-1 & -1 \\
1  & 1
\end{pmatrix}\mathrm{\bf y}.\end{equation}

It is easy to see that~\eqref{sys coeff} has form~\eqref{sys} with matrices $A,$ $C,$ and $F$ satisfying assumptions I--III, while $n=2$ and  $\rho(x) \equiv 1.$
As well, the matrix $C = \frac{\sigma^2}{2\lambda}\begin{pmatrix}
-1 & -1 \\
1  & 1
\end{pmatrix}$ satisfies condition~\eqref{C alpha}.
To prove this, we verify condition~\eqref{C alpha} for the matrices $C$ of a more general form.

\begin{prop} \label{C prop}
Let $C(x, \lambda)=\sum_{k=1}^{N} C_k(x)\lambda^{-k},$  where $N>0$ and $C_k,$ $k=\overline{1,N},$ are arbitrary matrices of order $n$ with the elements from $L[0, \infty).$
Then, the matrix $C$ satisfies~\eqref{C alpha}.
\end{prop}
\begin{proof}
For any $\alpha > 0,$ put $\displaystyle K_\alpha = \sum_{k=1}^{N} \| C_k\|_{L[\alpha, \infty)}.$ It is clear that $K_\alpha \to 0$ as $\alpha \to \infty.$ Then, to satisfy~\eqref{C alpha}, it is sufficient to take $\varphi(\alpha) = \max\{\alpha^{-1}, \sqrt[2N]{K_\alpha}\}.$ Indeed, we have
$$\sup_{|\lambda|\ge\varphi(\alpha)}\| C\|_{L[\alpha, \infty)} \le \sum_{k=1}^{N} \sup_{|\lambda|\ge\varphi(\alpha)}|\lambda^{-k}| \| C_k\|_{L[\alpha, \infty)} \le \sum_{k=1}^{N} \varphi^{-k}(\alpha)\| C_k\|_{L[\alpha, \infty)}.$$
One can see that $\varphi(\alpha) \to 0$ as $\alpha \to \infty,$ and $\varphi(\alpha) < 1$ for a sufficiently large $\alpha.$ Then, for such $\alpha,$
$$\sup_{|\lambda|\ge\varphi(\alpha)}\| C\|_{L[\alpha, \infty)} \le \sum_{k=1}^{N} \varphi^{-N}(\alpha) \| C_k\|_{L[\alpha, \infty)} = K_\alpha \varphi^{-N}(\alpha).$$
If $K_\alpha = 0,$ then $\| C\|_{L[\alpha, \infty)} = 0 \le \sqrt{K_\alpha}$ for any $\lambda \in \mathbb C_0.$ If $K_\alpha > 0,$ then $\varphi^{N}(\alpha) \ge \sqrt{K_\alpha}$ and $\displaystyle\sup_{|\lambda|\ge\varphi(\alpha)}\| C\|_{L[\alpha, \infty)} \le K_\alpha / \sqrt{K_\alpha} = \sqrt{K_\alpha}.$ In any case, $\displaystyle\sup_{|\lambda|\ge\varphi(\alpha)}\| C\|_{L[\alpha, \infty)} \le \sqrt{K_\alpha} \to 0$ as $\alpha \to \infty.$
\end{proof}

Thus, we can apply Theorem~\ref{fss} to system~\eqref{sys coeff}. 
We have $b_1 = 1,$ $b_2 = -1,$ and there are two sectors: the right-half plane $\Gamma_1 = \{ \lambda \in \mathbb C \colon {\mathrm Re}\, \lambda >0 \}$ and  the left-half plane $\Gamma_2 = \{ \lambda \in \mathbb C \colon {\mathrm Re}\, \lambda <0 \}.$ For $\alpha \ge 0$ and $\lambda \in \overline{\Gamma_1^\alpha},$ we obtain that system~\eqref{sys coeff} has a FSS $Y_\alpha=[y_{jk}(x, \lambda)]_{j,k=1}^2,$ whose components are continuous functions of $\lambda \in \overline{\Gamma_1^\alpha}$ and analytic functions of $\lambda \in {\Gamma_1^\alpha}.$ The following asymptotics hold:
\begin{equation} \label{asym1}
y_{jk}(x, \lambda) = \exp\left(b_k\lambda(x - \alpha)\right) \left(\delta_{jk} \exp\left((-1)^k i \int_\alpha^x \frac{p_0(t)}{2} \, dt\right) + s_{jk}(x, \lambda)\right), 
\end{equation}
where $s_{jk}(x, \lambda) = o(1)$ as $\lambda \to \infty$ uniformly on $x \ge \alpha,$ and $j,k = 1, 2.$ 
Since  $\frac{ip_0}{2}, \sigma \in L_2[0, \infty)$ and $\| C \|_{L[0, \infty)} = O(\lambda^{-1}),$
applying Theorem~\ref{complement}, we obtain 
\begin{equation} \label{ss}
\sup_{x \ge \alpha} |s_{jk}(x, \lambda)| \in L_2(\Sigma),\end{equation} 
where $\Sigma \subset \overline{\Gamma_1^\alpha}$ is an arbitrary half-line. 

Consider the matrix-function $V_\alpha = \Theta Y_\alpha,$ $V_\alpha =: [v_{jk}(x,\lambda)]_{j,k=1}^2,$ that is a FSS of system~\eqref{sys prev}.
Putting $u_1 := v_{11}$ and $u_2 := v_{12},$ from~\eqref{sys prev}, we have that $u_1$ and $u_2$ are solutions of equation~\eqref{pencil},
while $u^{[1]}_1 = \lambda v_{21}$ and $u^{[1]}_2 = \lambda v_{22}.$ 
For any $T>0$ and $k=1,2,$ we have $u_k(\cdot, \lambda), u_k^{[1]}(\cdot, \lambda) \in AC[0, T),$ since $v_{jk}(\cdot, \lambda)\in AC[0, T).$ 
The system $\{u_1, u_2 \}$ of solutions of equation~\eqref{pencil} is fundamental in the sense that the vectors 
$\Big[u_1, \; u^{[1]}_1\Big]^t$ and $\Big[u_2, \; u^{[1]}_2\Big]^t$ are linearly independent on $x \ge 0.$

Proceed to $z = -i \lambda,$ then $\lambda \in \overline{\Gamma_1}$ if and only if $z \in \overline{\mathbb C_-},$ where $\mathbb C_- = \{z \in \mathbb C \colon\mathrm{Im}\,z <0 \}.$ Using~\eqref{asym1} and~\eqref{ss}, we arrive at the following theorem.
 
\begin{theorem} \label{application}
For $\alpha \ge 0,$ there exists $\lambda_\alpha > 0$ such that for
$z \in \overline{\mathbb C_-^\alpha},$ $\mathbb C_-^\alpha := \{ z \in \mathbb C_- \colon |z| > \lambda_\alpha\},$ there exists a FSS of~\eqref{pencil} $\{ u_1(x, z), u_2(x, z) \}$ with the following properties.

\begin{enumerate}
\item
For $k=1,2,$ we have the representation
$$u_k(x, z) = \exp\big({(-1)^{k+1} z i(x - \alpha)}\big)\left(\exp\left({(-1)^{k} i \int_\alpha^x \frac{p_0(t)}{2} \, dt}\right) + \tilde s_{1k}(x, z)\right),$$
$$u^{[1]}_k(x, z) = (-1)^{k+1} zi \exp\big({(-1)^{k+1} z i(x - \alpha)}\big)\left(\exp\left({(-1)^{k} i \int_\alpha^x \frac{p_0(t)}{2} \, dt}\right)  + \tilde s_{2k}(x, z)\right),$$
where for $j=1,2,$ $\tilde s_{jk}(x, z) = o(1)$ uniformly on $x \ge \alpha$ as $z \to \infty.$  Moreover,  
$$\sup_{x \ge \alpha} |\tilde s_{jk}(x, \cdot)| \in L_2(\Sigma), \quad j,k =1,2,$$
where $\Sigma \subset \overline{\mathbb C_-^\alpha}$ is a half-line. 

\item  For a fixed $x  \ge 0,$ the functions $u_{jk}(x, z),$ $j,k=\overline{1, n},$ are continuous on $\overline{\mathbb C_-^\alpha}$ and  analytic in $\mathbb C_-^\alpha.$
\end{enumerate}
Moreover, $\lambda_\alpha \to 0$ as $\alpha \to \infty.$
\end{theorem}
An analogous theorem can be obtained for the values $z \in \overline{\mathbb C_+},$ $\mathbb C_+ := \{z\in \mathbb C \colon\mathrm{Im}\,z >0 \}.$ However, the difference between $z \in \mathbb C_+$ and $z \in \mathbb C_-$ is insignificant, since one case reduces to the other case by the substitution $z = -\tilde z.$ 

If $p_0 = 0,$ Theorem~\ref{application} gives the result for the Sturm--Liouville equation on the half-line with a distribution potential. It agrees with the formulae for FSS of $n$-th order equations given in~\cite{bond3,stud} in the particular case $n=2.$ We can also apply Theorem~\ref{application} to equations on a finite interval $x \in [0, T],$ putting $\sigma(x) = 0$ for $x > T.$ In this case, we obtain the FSS of the Sturm--Liouville equation on a finite interval similar to the ones provided in~\cite[Corollary~2]{sav}. 

 \medskip
\noindent{\bf Funding.}
This work was supported by the Russian Science Foundation (project no.~21-71-10001), \url{https://rscf.ru/en/project/21-71-10001/}.

\appendix

\section{Appendix: proof of Theorem~\ref{theta theorem}} \label{appendix A}

First, we introduce auxiliary values and prove necessary theorems for them. Let $f \in L[0, \infty)$ be some fixed function. Denote $\mathbb C_+ := \{\lambda \in \mathbb C \colon \mathrm{Im}\, \lambda > 0\}$ and
\begin{equation}\psi(s, x, \lambda) := 
\int_{\min\{s, x\}}^{\max\{s, x\}} f(t) e^{i \lambda |x - t|} \, dt, 
 \quad \Psi(\lambda) := \sup_{x,s \ge 0}|\psi(s, x, \lambda)|,\quad \lambda \in \overline{\mathbb C_+}. \label{psi}\end{equation}
 Note that the integrals in~\eqref{psi} are well-defined, since $\mathrm{Re}\,\big(i \lambda |x - t|\big) \le 0$ for $ \lambda \in \overline{\mathbb C_+}.$
 Obviously, $\Psi(\lambda)$ has real non-negative values.
Its definition is similar to the one of $\theta_\alpha(\lambda)$ and, in fact, $\Psi(\lambda)$ will participate in estimates for $\theta_\alpha(\lambda).$ Let us obtain necessary properties of  $\Psi(\lambda).$

\begin{prop} \label{Psi}
The function $\Psi(\lambda)$ is uniformly continuous on $\overline{\mathbb C_+}$ and $\Psi(\lambda) \to 0$ as $\lambda \to \infty.$
\end{prop}
\begin{proof}
First, note that for any $\zeta > 0,$ there exists $T > 0$ such that 
\begin{equation} \label{vartheta}
|\psi(s, x, \lambda)| \le \zeta + |\psi(\min\{s, T\}, \min\{ x, T\}, \lambda)|, \quad \lambda \in \overline{\mathbb C_+}, \; s, x \ge 0.
\end{equation}
Indeed, it is sufficient to take $T$ such that $\| f \|_{L[T, \infty)} \le \zeta.$ 
To prove~\eqref{vartheta}, for definiteness, we assume $x > s;$ the case $s \ge x$ can be considered analogously.
If $x \le T,$ then~\eqref{vartheta} is obvious. 
If $x > T,$ then we have
$$|\psi(s, x, \lambda)| = \left|\int_{s}^x f(t) e^{i \lambda (x - t)}\,dt \right| \le \left|\int_{\max\{T, s\}}^x f(t) e^{i \lambda (x - t)}\,dt \right| + \big|e^{i \lambda (x - T)}\big|\left|\int_{s}^T f(t) e^{i \lambda (T - t)} \, dt\right|,$$
where $\big|e^{i \lambda (x - T)}\big| \le 1$ and the second summand is absent if $s> T.$ Then, using the inequality $\| f \|_{L[T, \infty)} \le \zeta,$ we obtain $|\psi(s, x, \lambda)| \le \zeta + |\psi(\min\{s, T\}, T, \lambda)|,$ and~\eqref{vartheta} is proved. 

Now, we prove the continuity of $\Psi(\lambda).$ Let $\varepsilon > 0$ be fixed and $x_z, s_z > 0$ be such that
\begin{equation} \label{x_z}
|\psi(s_z, x_z, z)| +\frac\varepsilon2 \ge \Psi(z), \quad z \in \overline{\mathbb C_+}.
\end{equation}  
Using~\eqref{vartheta}, we can always choose $x_z, s_z \in [0, T],$ where $T>0$ depends on $\varepsilon.$ Let $\delta > 0$ be such that 
\begin{equation} \label{uniform e}
\big|e^{\lambda i \xi} - e^{z i \xi}\big| \le \frac{\varepsilon}{2\| f\|_{L[0, \infty)}+1}, \quad \xi \in [0, T],
\end{equation}
as soon as $\lambda, z \in \overline{\mathbb C_+}$ and $|\lambda -z| < \delta.$  
Then, for such $\lambda$ and $z,$ we have
$$\Psi(\lambda) - \Psi(z) \overset{(\ref{x_z})}{\ge}  |\psi(s_z, x_z, \lambda)|
- |\psi(s_z, x_z, z)| - \frac\varepsilon2 \ge - \frac\varepsilon2 - \int \left|e^{\lambda i |t - x_z|} - e^{z i |t - x_z|}\right| |f(t)| \, dt,$$
wherein the integral is taken on a subinterval of $[0, T],$ and $|t - x_z| \in [0, T].$ Using~\eqref{uniform e}, we arrive at $\Psi(\lambda) - \Psi(z) \ge -\varepsilon.$ 
Note that $\lambda$ and $z$ can be swapped. Then, for any $\varepsilon > 0$ there exists $\delta > 0$ such that $|\Psi(\lambda) - \Psi(z)| \le \varepsilon$ for $|\lambda -z| < \delta,$ and the continuity is proved.

The statement $\Psi(\lambda) \to 0$ as $\lambda \to \infty$ is proved analogously to that $\theta_\alpha(\lambda) \to 0$ in Lemma~\ref{theta lemma}. 
\end{proof}

\begin{theoremA} \label{A} If $f \in L_2[0, \infty),$ 
 then $\Psi \in L_2(\mathbb R).$
 \end{theoremA}
\begin{proof}
Consider the Carleson operator ${\cal C},$ that acts on $w \in L_2(\mathbb R)$ as follows:
$${\cal C}w(\sigma) := \displaystyle\mathop{\sup}_{N > 0} \Big| \int_{-N}^N \hat{w}(s) e^{2 \pi i \sigma s}\, ds \Big|, \quad \sigma \in \mathbb R,$$
where $\hat w$ is the Fourier transform of $w:$ 
$$\displaystyle \hat w(s) = \mathop{\int}_{\mathbb R} w(x) e^{-2 \pi i s x} \,dx, \quad s \in \mathbb R.$$
 It is known that ${\cal C}$ is a bounded operator in $L_2(\mathbb R)$ (see, e.g., Theorem~6.2.1 in~\cite{grafakos}).

Put $f(t) = 0$ for $t < 0.$ There is defined the inverse Fourier transform of $f:$ 
$\check f(t) = \hat f(-t),$ $t\in \mathbb R,$ and $\check{f} \in L_2(\mathbb R).$
Since $({\check{f}})\hat{\vphantom{1mm}}  = f,$ we have
\begin{equation} \label{Lr}
\mathop{\sup}_{N > 0} \left| \int_{0}^N f(s) e^{2 \pi i \sigma s}\, ds \right| = {\cal C}\check{f}(\sigma) \in L_2(\mathbb R).
\end{equation}

Note that $\Psi(\lambda) = \max\{\Psi_1(\lambda),\Psi_2(\lambda)\},$ where 
$$\Psi_1(\lambda) := \sup_{x \ge s \ge 0}\left|\int_{s}^x f(t) e^{i \lambda (x -t)}\,dt\right|, \quad \Psi_2(\lambda) := \sup_{s \ge x \ge 0}\left|\int_{x}^s f(t) e^{i \lambda (t -x)}\,dt\right|.$$ 
Then, it is sufficient to obtain $\Psi_1, \Psi_2 \in L_2(\mathbb R).$ 
Consider, for definiteness, $\Psi_1.$ 
For $\tau \in \mathbb R,$ we have
$$ \Psi_1(-\tau) = \sup_{x \ge s \ge 0}\left|\int_{s}^x f(t) e^{i \tau t}\,dt\right|\big|e^{-i \tau x}\big|=\sup_{x \ge s \ge 0}\left|\int_{s}^x f(t) e^{i \tau t}\,dt\right|.$$
Applying the inequality $|\int_s^x| \le |\int_0^x|+|\int_0^s|,$ we arrive at
$$\Psi_1(-\tau) \le 2 \sup_{N>0} \left|\int_{0}^N f(t) e^{i \tau t}\,dt\right|=2 {\cal C}\check{f}\left(\frac{\tau}{2\pi}\right),$$
which along with~\eqref{Lr} give us the needed statement. For $\Psi_2,$ the computations are analogous.
\end{proof}

Further, we need the following definition.

\begin{definition}
A measure $\mu \ge 0$ in $\mathbb C_+$ is called Carleson measure if 
$$\sup_{x \in \mathbb R, \; y > 0} y^{-1} \mu(S_{x, y}) < \infty, \quad S_{x,y} := \big\{ \lambda \in \mathbb{C} \colon \mathrm{Re} \, \lambda \in (x, x+y), \;  \mathrm{Im} \, \lambda \in (0, y)\big\}.$$ 

In particular, if $\mu$ is the standard Lebesgue measure on a half-line $\tilde \Sigma \subset \mathbb C_+$ that equals $0$ in the other part of the plane $\mathbb C_+,$ then $\mu$ is a Carleson measure.
\end{definition}

For a measure $\mu \ge 0$ in $\mathbb C_+,$ denote by $L_2(\mathbb C_+, \mu)$ the set of measurable functions $h(z),$ $z \in \mathbb C_+,$ for which the Lebesgue integral $\int_{\mathbb C_+} |h(z)|^2 \, d\mu$ is finite.

\begin{theoremA} \label{B}
Under the conditions of Theorem~\ref{A},
for any Carleson measure $\mu,$ $
\Psi \in {L_2(\mathbb C_+, \mu)}.$
\end{theoremA}
To prove Theorem~\ref{B}, we need several auxiliary statements.

\begin{theoremA}[Theorem~5.6 in \cite{garnett}] \label{C}
 Let $g$ be an arbitrary function from $L_2(\mathbb R)$ and  let 
$\Phi(z),$ $z \in \mathbb C_+,$ be its Poisson integral:  
\begin{equation} \label{poisson}
\Phi(x + i y) = \mathop{\int}_{\mathbb R} \frac{y}{\pi([x-t]^2 + y^2)} g(t) \, dt, \quad y > 0,\;x \in \mathbb R.
\end{equation}
Then, for any Carleson measure $\mu,$ we have  $\Phi \in L_2(\mathbb C_+, \mu).$ 
\end{theoremA}

\begin{theoremA}[Corollary~3.2 in \cite{garnett}] \label{D} Let $g$ be a bounded and uniformly continuous function on $\mathbb R.$ Consider the function
$\Phi$ defined by~\eqref{poisson} and extend it on $x \in \mathbb R$ by the formula $\Phi(x) = g(x).$
Then, the function $\Phi$ is harmonic in $\mathbb C_+$ and continuous on $\overline{\mathbb C_+}.$
\end{theoremA}

\begin{definition}[\cite{conway, garnett}]
Let $w \colon \Omega \to \mathbb R$ be a continuous function, where $\Omega \subseteq \mathbb C$ is a domain. 
For $\lambda \in \mathbb C$ and $\delta >0,$ denote $B_\delta(\lambda) = \{z \in \mathbb C \colon |z - \lambda| < \delta \}.$
The function $w$ is called subharmonic if 
$$w(\lambda) \le \frac{1}{2\pi} \int_0^{2\pi} w(\lambda + \delta e^{it}) \, dt $$
whenever $\overline{B_\delta(\lambda)} \subset \Omega.$
Clearly, every harmonic function is subharmonic.
\end{definition}

\begin{theoremA}[Maximum Principle, see \S3.3 in \cite{conway}]
\label{E}  Let $w_1$ and $w_2$ be real-valued continuous functions on a domain $\Omega \subseteq \mathbb C.$
Denote $\partial_\infty \Omega = \partial \Omega$ if $\Omega$ is bounded and $\partial_\infty \Omega = \partial \Omega\cup\{\infty\}$ if $\Omega$ is unbounded. Assume that $w_1$ and $-w_2$ are subharmonic and that 
$$\mathop{\lim \sup}_{\substack{z \to \lambda, \\ z \in \Omega}} w_1(z) \le \mathop{\lim \inf}_{\substack{z \to \lambda, \\ z \in \Omega}} w_2(z), \quad \lambda \in \partial_\infty \Omega.$$
Then, $w_1(z) \le w_2(z)$ for $z \in \Omega.$
\end{theoremA}

\begin{proof}[Proof of Theorem~\ref{B}]
Applying Gauss' mean value theorem to $\psi(s, x, \lambda)$ in $\lambda \in \mathbb C_+$ and proceeding to the supremum by $x,s \ge 0,$ 
we obtain $$\Psi(\lambda) \le \frac{1}{2 \pi} \int_0^{2\pi} \Psi(\lambda + \delta e^{i t}) \, dt, \quad \lambda \in \mathbb C_+, \quad \delta \in (0, \mathrm{Im}\, \lambda).$$ Then, the function $\Psi$ is subharmonic in $\mathbb C_+.$ 

Define $\Phi(z)$ for $z \in \mathbb C_+$ by~\eqref{poisson} with $g =\Psi$
and put $\Phi(z) = \Psi(z)$ for $z \in \mathbb R.$ 
Proposition~\ref{Psi} yields that $\Psi(z)$ is uniformly continuous and bounded on $\mathbb R.$ 
Then, by Theorem~\ref{D}, the function $\Phi$ is harmonic on $\mathbb C_+$ and continuous on $\overline{\mathbb C_+}.$ 
Since $\displaystyle\lim_{z \to \infty} \Psi(z) = 0,$ $\Phi(z) \ge 0,$ and for any $x \in \mathbb R,$ $\displaystyle\lim_{z \to x} \Psi(z) = \displaystyle\lim_{z \to x} \Phi(z),$ we can apply Theorem~\ref{E} with $w_1 = \Psi$ and $w_2 = \Phi.$
Then, we arrive at $\Psi(\lambda) \le \Phi(\lambda)$ for $\lambda \in \mathbb C_+.$  

By virtue of Theorem~\ref{A}, we also have $\Psi \in L_2(\mathbb R).$ Applying Theorem~\ref{C}, we obtain $\Phi \in L_2(\mathbb C_+, \mu),$ and, consequently, $\Psi \in L_2(\mathbb C_+, \mu).$ 
\end{proof}

Now, we are ready to prove Theorem~\ref{theta theorem}.
\begin{proof}[Proof of Theorem~\ref{theta theorem}]
For each pair of indices $j,l = \overline{1, n},$ we prove that $\displaystyle\sup_{x, s \ge \alpha} |\nu_{jl}(s, x, \lambda)| \in L_2(\Sigma).$ We can exclude from consideration the pairs with $b_l = b_j,$ because for them  $\nu_{jl} \equiv q_{jl} \equiv 0.$
For the rest values $j, l$ and $\lambda \in \overline{\Lambda},$ under condition~\eqref{omega numer}, we obtain the estimate 
\begin{equation} \label{tilde lambda}
\sup_{x, s \ge \alpha} |\nu_{jl}(s, x, \lambda)| \le \Psi_{jl}(\tilde \lambda), \quad 
\overline{\mathbb C_+} \ni 
\tilde \lambda = \left\{\begin{array}{cc}
-i \lambda (b_j - b_l), & l < j,\\
i \lambda (b_j - b_l), & j < l,
\end{array}\right.
\end{equation}
where $\Psi_{jl}(\lambda)$ is constructed as the function $\Psi(\lambda)$ in~\eqref{psi} with 
$f(\xi) := \frac{q_{jl}(p_\alpha^{-1}(\xi))}{\rho(p_\alpha^{-1}(\xi))},$ while $p_\alpha^{-1}$ is the inverse function to $p_\alpha(x) := p(x) - p(\alpha).$ Note that under our assumptions on $a_{jl}$ and $\rho,$ $f \in L\cap L_2[0, \infty).$

To prove~\eqref{tilde lambda},
one should consider all possible positions of $j$ and $l$ with respect to the fixed $k \in \overline{1, n}$ as listed in~\eqref{rule}. 
For definiteness, let $j, l < k;$ the other cases are proceeded analogously (see~\cite[Theorem~3.1]{sav-calderon}). Then, we can assume $x < s$ (otherwise $\nu_{jl} \equiv 0$). Consider the subcase $j < l.$
We can rewrite
$$\nu_{jl}(s, x, \lambda) = e^{\lambda(b_l - \omega)(p_\alpha(x) - p_\alpha(s))} \int_x^s q_{jl}(t) e^{\tilde \lambda i (p_\alpha(t) - p_\alpha(x))}\,dt, \quad \big|e^{\lambda(b_l - \omega)(p_\alpha(x) - p_\alpha(s))}\big| \overset{(\ref{omega numer})}{\le} 1.
$$ 
After change of the variables $t = p_\alpha^{-1}(\xi),$ $x = p_\alpha^{-1}(\zeta),$ and $s = p_\alpha^{-1}(\eta),$ we arrive at
$$\sup_{s, x \ge \alpha} |\nu_{jl}(s, x, \lambda)| \le \sup_{\eta > \zeta \ge 0}\left|\int_\zeta^\eta f(\xi) e^{\tilde \lambda i(\xi - \zeta)} \, d\xi\right| \le \Psi_{jl}(\tilde \lambda).$$
Consider the subcase $l < j,$ in which we use the representation 
$$\nu_{jl}(s, x, \lambda) = e^{\lambda(b_j - \omega)(p_\alpha(x) - p_\alpha(s))} \int_x^s q_{jl}(t) e^{\tilde \lambda i (p_\alpha(s) - p_\alpha(t))}\,dt, \quad \big|e^{\lambda(b_j - \omega)(p_\alpha(x) - p_\alpha(s))}\big| \overset{(\ref{omega numer})}{\le} 1.
$$ 
After change of the variables $t = p_\alpha^{-1}(\xi),$ $x = p_\alpha^{-1}(\zeta),$ and $s = p_\alpha^{-1}(\eta),$ we arrive at
$$\sup_{s, x \ge \alpha} |\nu_{jl}(s, x, \lambda)| \le \sup_{\eta > \zeta \ge 0}\left|\int_\zeta^\eta f(\xi) e^{\tilde \lambda i(\eta-\xi)} \, d\xi\right| \le \Psi_{jl}(\tilde \lambda).$$

Thus, we have obtained~\eqref{tilde lambda}. Denote by ${\tilde \Sigma}$ the image of the half-line $\Sigma$ after action of the linear mapping $\lambda \to \tilde \lambda.$ 
Then, ${\tilde \Sigma} \subset \overline{\mathbb C_+}$ is a half-line, and it remains to prove that $\Psi_{ij}(\tilde \lambda) \in L_2({\tilde \Sigma}).$ 
 If $\tilde\Sigma \subset \mathbb R,$ then the desired statement follows from Theorem~\ref{A}. If $\tilde\Sigma \subset \mathbb C_+,$ then we use the representation $\int_{\tilde \Sigma}  |\Psi_{ij}(\tilde \lambda)|^2 \, d\tilde\lambda = \int_{\mathbb C_+}  |\Psi_{ij}(\tilde \lambda)|^2 \, d\mu,$ where $\mu$ is a Carleson measure, and apply Theorem~\ref{B}.
\end{proof}

\section{Appendix: holomorphic mappings} \label{appendix B}
Here, we provide some basic notions and theorems from the theory of holomorphic mappings, see~\cite{hol}. 
We also obtain Proposition~\ref{hol} that is used for proving the holomorphy of the function~\eqref{sol rep} constructed in Lemma~\ref{main lemma}.

Let $\bf X$ be a complex Banach space with the norm $\| \cdot\|_{\bf X}$ and let $\Omega \subseteq \mathbb C$ be a domain. 
\begin{definition}
A mapping $G\colon  \Omega \to \bf X$ is called holomorphic in $\Omega$ if it is differentiable, i.e.  
$$\forall\lambda \in \Omega\ \exists\dot{G}(\lambda) \in {\bf X} \colon \quad \lim_{\xi \to \lambda}\left\|\frac{G(\xi) - G(\lambda)}{\xi - \lambda} - \dot{G}(\lambda)\right\|_{\bf X} = 0.$$
In this case, the mapping $\dot G \colon  \Omega \to \bf X$ is called the derivative of $G.$ 
\end{definition}
It is well-known that holomorphy is a stronger property than continuity (a mapping $G\colon \Omega \to \bf X$ is continuous  if  $\|G(\xi) - G(\lambda)\|_{\bf X} \to 0$ as $\xi \to \lambda \in \Omega$). Holomorphic mappings are the objects that generalize holomorphic functions, and in general, they have similar properties, see~\cite{hol}. Obviously, a linear combination of holomorphic mappings is a holomorphic mapping. 

The following theorem is an analogue of the statement that the uniform limit of holomorphic functions is also holomorphic.

\begin{theoremA}[9D in~\cite{hol}] \label{lim hol maps}
Let $G_n \colon \Omega \to {\bf X},$ $n \ge 0,$ be a sequence of holomorphic mappings. 
If $G_n$ converges uniformly on each compact subset of $\Omega$ to the mapping $G\colon \Omega \to {\bf X},$ then $G$ is holomorphic.
\end{theoremA}

We also need two theorems that have no analogues in classical complex analysis.
\begin{theoremA}[12C in~\cite{hol}] \label{G}
Let $G \colon \Omega \to {\bf X},$ where ${\bf X}$ is the Cartesian product of complex Banach spaces ${\bf X}_1, \ldots, {\bf X}_m.$ Let $G_k \colon \Omega \to {\bf X}_k$ be the coordinate map of $G,$ $k = \overline{1, m}.$
Then, $G$ is holomorphic in $\Omega$ if and only if  $G_k$ is  holomorphic in $\Omega,$ $k=\overline{1, m}.$ 
\end{theoremA}

\begin{theoremA}[9.12 in~\cite{hol}] \label{H}
Denote by ${\bf X}^*$ the set of all linear bounded functionals on ${\bf X}.$ 
A mapping $G \colon \Omega \to {\bf X}$ is holomorphic if and only if  for any $f \in {\bf X}^*,$ $f(G(\lambda))$ is a holomorphic function in $\Omega.$ 
\end{theoremA}

In particular, these theorems yield that if $\mathrm{\bf z} = [z_j(x, \lambda)]_{j=1}^n$ is a holomorphic mapping of $\lambda \in \Omega \to {\bf BC}_{n},$ then for $j \in \overline{1, n}$ and any fixed $x \ge \alpha,$ $z_j(x, \lambda)$ is a holomorphic function in $\Omega.$ The analyticity of the components $y_{jk}(x, \lambda)$ and $u_{jk}(x, \lambda)$ in Theorems~\ref{fss} and~\ref{neighboring} will follow from this fact.


We aim to prove the following proposition.
\begin{prop} \label{hol}
Consider the operator ${\cal V}_k(\lambda)$ and the domain $\Lambda^\alpha$ determined in Lemma~\ref{main lemma}. Let $\mathrm{\bf w} \in {\bf BC}_n$ satisfy the conditions in item 3) of Lemma~\ref{main lemma}.
Then, the mapping ${\cal V}_k(\lambda)\mathrm{\bf w} \colon \Lambda^\alpha \to {\bf BC}_n$ is holomorphic.
\end{prop}
For the proof, we need a lemma.
\begin{lemma} \label{delta lemma}
Let ${\mathrm Z} \subset \mathbb C$ and $\beta \in \mathbb{C}_0$ be such that $\mathrm{Re}\,\big(\xi\beta\big) \ge 0$ for $\xi \in \overline{{\mathrm Z}}.$ Then, for any $\delta >0$ there exists $\varepsilon>0$ such that for $\xi \in {\mathrm Z}_{\delta} := \{ \xi \in {\mathrm Z} \colon \mathrm{dist}\,(\xi, \partial {\mathrm Z}) \ge \delta\},$ we have $\mathrm{Re}\,\big(\xi\beta\big) \ge \varepsilon.$
\end{lemma}
\begin{proof}
First, we prove the statement in the case $\beta=1.$
We have that ${\mathrm Z}$ is a subset of the closed right half-plane $\overline{\mathbb C_r} = \{ \lambda \in \mathbb C \colon  \mathrm{Re}\,\lambda \ge 0\}.$ 
Note that for arbitrary sets $A$ and $B$ in a normed space, $\mathrm{dist}(\xi, \partial B) \ge \mathrm{dist}(\xi, \partial A)$ as soon as $A \subseteq B$ and $\xi \in A.$ Applying this inequality to $A = {\mathrm Z}$ and $B = \overline{\mathbb C_r},$ we obtain
$$\mathrm{Re}\,\xi = \mathrm{dist}(\xi, \partial \mathbb C_r) \ge \mathrm{dist}(\xi, \partial {\mathrm Z}) \ge \delta, \quad \xi \in {\mathrm Z}_\delta.$$ 
Thus, we arrive at the needed statement with $\varepsilon=\delta.$

The case $\beta \ne 1$ reduces to the previous one by the substitution $\xi \beta =: \tilde \xi.$ 
\end{proof}
\begin{proof}[Proof of Proposition~\ref{hol}.]
Denote $\mathrm{\bf w} =: [w_j(x, \lambda)]_{j=1}^n$ and ${\cal V}_k(\lambda)\mathrm{\bf w} =: [h_j(x, \lambda)]_{j=1}^n.$
By Theorem~\ref{G}, we should prove that for each  $j\in \overline{1, n},$ the component $h_j(x, \lambda)$ is a holomorphic mapping of $\lambda \in \Lambda^\alpha \to {\bf BC}_1.$
 Remind that ${\bf BC}_1$ denotes the Banach space of the functions that are bounded and continuous on $[\alpha, \infty)$ having the norm $\displaystyle\| \psi \|_{{\bf BC}_1} = \sup_{x \ge \alpha} |\psi(x)|.$ 
  
For definiteness, assume $j < k$ (the case $j \ge k$ is analogous). 
Then, denoting $\beta = b_j - \omega$ and $f(t, \lambda) = -\sum_{l=1}^n v_{jl}(t, \lambda) w_l(t, \lambda),$ we have
\begin{equation*} \label{hhh}
h_j(x, \lambda) = \int_x^\infty f(t, \lambda) e^{\lambda \beta (p(x) - p(t))}\, dt.
\end{equation*}
Under our assumptions on the matrix $C(x, \lambda),$ the elements $v_{jl}$  are holomorphic mappings of $\lambda \in \mathbb C_0 \to L[0, \infty).$  By direct differentiation, one can check that $f \colon \lambda \in \Lambda^\alpha \to L[\alpha, \infty)$ is a holomorphic mapping.

By~\eqref{omega numer}, for $\xi \in \Lambda^\alpha,$ we have $\mathrm{Re}\,\big(\xi \beta\big) \ge 0,$ and
\begin{equation} \label{weak estimate}
|e^{-\xi \beta}|\le 1, \quad \xi \in \Lambda^\alpha.
\end{equation}
Let $\lambda \in \Lambda^\alpha$ be fixed and let $\delta > 0$ be such that $\overline{B_{2\delta}(\lambda)}\subset \Lambda^\alpha,$ where we denoted $B_{r}(\lambda) = \{\xi \in \mathbb C \colon |\lambda - \xi| < r\},$ $r > 0.$ 
If $\beta \ne 0,$ we can apply Lemma~\ref{delta lemma} to ${\mathrm Z} := \Lambda^\alpha.$
Then, since $\mathrm{dist}\,(\xi, \partial \Lambda^\alpha) \ge \delta$ for $\xi \in B_{\delta}(\lambda),$ we have
 \begin{equation*} 
\left|e^{-\xi \beta}\right| \le e^{-\varepsilon}, \quad \xi \in B_{\delta}(\lambda), \quad \beta \ne 0,
\end{equation*}
where $\varepsilon > 0$ is the constant determined in Lemma~\ref{delta lemma} that depends on $\delta$ and $\beta.$
Computing the maximal value of the function $z e^{-\varepsilon z},$ $z \ge 0,$ we obtain another inequality:
\begin{equation} \label{strong estimate}
\left|z e^{-\xi \beta z}\right| \le z e^{-\varepsilon z} \le \frac{1}{\varepsilon e}, \quad z \ge 0, \; \xi \in B_{\delta}(\lambda), \quad  \beta \ne 0.
\end{equation}

Consider the function 
\begin{equation*} \label{hh}
\dot h_j(x, \lambda) = \int_x^\infty \dot f(t, \lambda) e^{\lambda \beta (p(x) - p(t))} \, dt + \beta\int_x^\infty f(t, \lambda) (p(x) - p(t)) e^{\lambda \beta (p(x) - p(t))} \, dt.
\end{equation*}
Inequalities~\eqref{weak estimate} and~\eqref{strong estimate} yield that $\dot h_j(x, \lambda) \in {\bf BC}_1,$ 
since $\dot f(\cdot, \lambda)$ and $f(\cdot, \lambda)$ belong to $L[\alpha, \infty).$ We prove that 
\begin{equation}\label{limit}
\lim_{\xi\to\lambda} \left\| \frac{h_j(\cdot, \xi) - h_j(\cdot, \lambda)}{\xi-\lambda} - \dot h_j(\cdot, \lambda) \right\|_{{\bf BC}_1} = 0.
\end{equation}
Consider a splitting of the expression under the norm $\| \cdot\|_{{\bf BC}_1}:$ 
\begin{multline} \label{splitting}
 \frac{h_j(x, \xi) - h_j(x, \lambda)}{\xi - \lambda} - \dot{h_j}(x, \lambda)
 = \int_x^\infty \left[\frac{f(t, \xi) - f(t,\lambda)}{\xi - \lambda} - \dot f(t, \lambda)\right] e^{-\xi \beta z(x, t)}\,dt + \\+ \int_x^\infty \dot f(t, \lambda) \big[e^{-\xi \beta z(x, t)} - e^{-\lambda \beta z(x, t)}\big]\,dt 
 + \int_x^\infty f(t, \lambda) \eta(x, t, \lambda, \xi)\,dt,
\end{multline}
where 
$$z(x, t) = p(t) - p(x)\ge 0, \quad \eta(x, t, \lambda, \xi) =\frac{e^{-\xi \beta z(x, t)} - e^{\lambda \beta z(x, t)}}{\xi - \lambda} + \beta z(x, t)e^{-\lambda \beta z(x, t)}.$$
Using~\eqref{weak estimate},~\eqref{strong estimate}, and the holomorphy property of $f(t, \lambda),$ one can prove that each of the three integrals  in~\eqref{splitting} tends to $0$  uniformly on $x \ge \alpha$ as $B_\delta(\lambda) \ni \xi \to \lambda.$ The proofs for the first and for the second integrals are more simple than the proof for the third one. We provide the proof for the third integral if $\beta \ne 0$ (otherwise it immediately turns $0$).

Let $\varepsilon_0 > 0$ be arbitrary and the constants $N, T>0$ be such that  $\| f(\cdot, \lambda)\|_{L[\alpha, \infty)} < N$ and $\| f(\cdot, \lambda)\|_{L[T, \infty)} < \varepsilon_0\varepsilon/(4|\beta|).$ Let us find $\delta_0 \in (0, \delta)$ such that for $\xi \in B_{\delta_0}(\lambda)$ and $x \ge \alpha,$ there holds the estimate $\left|\int_x^\infty f(t, \lambda) \eta(x, t, \lambda, \xi)\,dt\right| \le \varepsilon_0.$
Consider the inequality
\begin{multline} \label{unif int}
\left|\int_x^\infty f(t, \lambda) \eta(x, t, \lambda, \xi)\,dt\right|  \le \\
 \int_x^{\max\{x, T\}} |f(t, \lambda)| |\eta(x, t, \lambda, \xi)|\,dt +  \int_{\max\{x,T\}}^\infty |f(t, \lambda)| |\eta(x, t, \lambda, \xi)|\,dt.
\end{multline}
Note that $\eta(x, t, \lambda, \xi) \to 0$ uniformly on  $x, t  \in [\alpha , T]$ as $\xi \to \lambda.$  Choose $\delta_0 \in (0, \delta)$ such that for $\xi \in B_{\delta_0}(\lambda)$ and $x, t  \in [\alpha , T],$ $|\eta(x, t, \lambda, \xi)| \le \varepsilon_0 /(2N).$  If $x \le T,$ we have 
\begin{equation} \label{ineq eta}
\int_x^{\max\{x, T\}} |f(t, \lambda)| |\eta(x, t, \lambda, \xi)|\,dt \le \frac{\varepsilon_0}{2}, \quad \xi \in B_{\delta_0}(\lambda).
\end{equation}
 If $x > T,$ then $\int_x^x = 0,$ and~\eqref{ineq eta} holds as well.

For $t \ge \max\{x, T\}$ and $\xi \in B_{\delta_0}(\lambda),$ we apply the inequality
$$\left| \frac{e^{-\xi \beta z(x, t)} - e^{-\lambda\beta z(x, t)}}{\xi - \lambda}\right| \le \max_{\nu \in \overline{B_{\delta_0}(\lambda)}}|\beta z(x, t) e^{-\nu \beta z(x, t)}|,$$
being a consequence of the fundamental theorem of calculus for segments in the complex plane.
Using this inequality along with~\eqref{strong estimate}, for $\xi \in B_{\delta_0}(\lambda),$ we obtain
\begin{multline} \label{ineq eta2}
\int_{\max\{x, T\}}^\infty |f(t, \lambda)| |e(x, t, \lambda, \nu)| \, dt \le 
2 |\beta| \int_{\max\{x, T\}}^\infty |f(t, \lambda)| \max_{\nu \in \overline{B_{\delta_0}(\lambda)}}|z(x,t) e^{-\nu \beta z(x,t)}|\, dt \le\\
\le \frac{2 |\beta|}{\varepsilon e} \int_T^\infty  |f(t, \lambda)| \,dt < \frac{\varepsilon_0}{2}.
\end{multline}
Combining the estimates~\eqref{unif int}--\eqref{ineq eta2}, we get that  $\left|\int_x^\infty f(t, \lambda) \eta(x, t, \lambda, \xi)\,dt\right| \le \varepsilon_0,$ $\xi \in B_{\delta_0}(\lambda),$ $x \ge \alpha.$ 
Thus, the third integral in~\eqref{splitting} tends to $0$ uniformly on $x.$ Analogously, the other two integrals in~\eqref{splitting} tend to $0$ uniformly on $x.$ Then, we arrive at relation~\eqref{limit}, which means holomorphy at $\lambda \in \Lambda^\alpha.$
\end{proof}

\begin{remark} We consider systems~\eqref{sys} with matrices $C(x, \lambda)$ whose components are holomorphic mappings of $\lambda \in \mathbb C_0$ to $L[0, \infty)$ under condition~\eqref{C con}. One can show that the components with the mentioned properties 
are expanded into the Laurent series:
\begin{equation} \label{cc}
c_{jk}(x, \lambda) = \sum_{\eta=1}^\infty \frac{c_{jk\eta}(x)}{\lambda^\eta}, \quad \lambda \in \mathbb{C}_0, \quad \| c_{jk\eta}\|_{L[0, \infty)} \le \delta^\eta N_\delta , \; \eta \ge 1, \; j,k = \overline{1, n},
\end{equation}
where $\delta >0$ can be chosen arbitrary and $N_\delta > 0.$ In~\cite{stud}, representation~\eqref{cc} was used as a restriction on the coefficients of $C(x, \lambda).$  

To prove~\eqref{cc}, consider a holomorphic mapping $G \colon \mathbb C_0 \to \bf X$ such that $\displaystyle \lim_{\lambda \to \infty} \| G(\lambda)\|_{\bf X} = 0.$
Denote $\xi = \frac{1}{\lambda} \in \mathbb C_0$ and $\tilde G(\xi) = G(\lambda).$ Then, $\tilde G \colon \xi \in \mathbb C_0 \to \bf X$ is holomorphic and $\displaystyle \lim_{\xi \to 0} \big\| \tilde G(\xi)\big\|_{\bf X} = 0.$
Let us apply the criterion from Theorem~\ref{H}. For any $f \in {\bf X}^*,$ we have that $f\big(\tilde G(\xi)\big)$ is a holomorphic function in $\mathbb C_0$ and $\displaystyle \lim_{\xi \to 0}  f\big(\tilde G(\xi)\big) = 0.$ Put $f\big(\tilde G(0)\big) = 0.$ 
Then, after removing singularity at $\xi = 0,$ the function $f \big(\tilde G(\xi)\big)$ is holomorphic in  $\mathbb C.$

By Theorem~\ref{H}, the mapping $\tilde G$ with the value $\tilde G(0) \equiv 0$ is holomorphic in $\mathbb C.$
It is expanded into the Taylor series, see~\cite[\S9.16--9.17]{hol}:
$$\tilde G(\xi)  = \sum_{\eta = 0}^\infty g_\eta \xi^\eta, \quad \| g_\eta\|_{\bf X} \le \delta^\eta N_\delta, \; \eta \ge 0,$$
wherein $g_0 = \tilde G(0) \equiv 0.$ Substituting $\xi = 1/\lambda,$ for $\tilde G(\xi) = G(\lambda) :=  c_{jk}(x, \lambda),$ we arrive at~\eqref{cc}.

\end{remark}
\end{document}